\documentclass[11pt]{amsart}
\usepackage{amsmath, amssymb, amsfonts, verbatim, graphicx, amsthm, tikz-cd, mathrsfs, color, comment, fancyhdr, amsrefs}
\usepackage[top=1in, bottom=1in, left=1in, right=1in]{geometry}
\usepackage{enumerate}

\usepackage{mathtools}

\usepackage{url}

\usepackage{enumitem}

\usepackage[symbol]{footmisc}

\usepackage[foot]{amsaddr}

\allowdisplaybreaks[1]

\usepackage[sc]{mathpazo}
\usepackage[T1]{fontenc}

\newcommand{\vertiii}[1]{{\left\vert\kern-0.25ex\left\vert\kern-0.25ex\left\vert #1
		\right\vert\kern-0.25ex\right\vert\kern-0.25ex\right\vert}}

\theoremstyle{plain}
\newtheorem{Thm}{Theorem}[section]
\newtheorem{Prop}[Thm]{Proposition}
\newtheorem{Lem}[Thm]{Lemma}
\newtheorem{Cor}[Thm]{Corollary}

\newtheorem{Qn}[Thm]{Question}

\newtheorem{Main results}[Thm]{Main results}

\newtheorem*{W-W}{Wiener-Wintner pointwise ergodic theorem}

\theoremstyle{definition}
\newtheorem{Def}[Thm]{Definition}
\newtheorem{Not}[Thm]{Notation}
\newtheorem{Rmk}[Thm]{Remark}

\renewcommand{\epsilon}{\varepsilon}

\usepackage{setspace}

\usepackage{hyperref}

\title{Noncommutative Ergodic Optimization}

\author{Aidan Young$^1$}
\address{University of North Carolina at Chapel Hill}

\email{$^1$\url{aidanjy@live.unc.edu}}
\begin{document}
	
	\maketitle
	
\begin{abstract}
We extend the theory of ergodic optimization and maximizing measures to the non-commutative field of C*-dynamical systems. We then provide a result linking the ergodic optimizations of elements of a C*-dynamical system to the convergence of certain ergodic averages in a suitable seminorm. We also provide alternate proofs of several results in this article using the tools of nonstandard analysis.
\end{abstract}

One of the guiding questions of the field of ergodic optimization is the following: Given a topological dynamical system $(X, G, U)$, and a real-valued continuous function $f \in C(X)$, what values can $\int f \mathrm{d} \mu$ take when $\mu$ is an invariant Borel probability measure on $X$, and in particular, what are the extreme values it can take? In a joint work with I. Assani \cite[Section 3]{Assani-Young}, and later in \cite{PointwiseReductionHeuristic}, we noticed that the field of ergodic optimization was relevant to the study of certain temporo-spatial differentiation problems. Hoping to extend these tools to the study of temporo-spatial differentiation problems in the setting of operator-algebraic dynamical systems, this article develops an operator-algebraic formalization of this question of ergodic optimization, re-interpreting it as a question about the values of invariant states on a C*-dynamical system.

Section \ref{Ergodic optimization} develops the theory of ergodic optimization in the context of C*-dynamical systems, where the role of ``maximizing measures" is instead played by invariant states on a C*-algebra. The framework we adopt is in fact somewhat more general than the classical framework of maximizing measures, since we consider ergodic optimizations relative to a restricted class of invariant states, which we call relative ergodic optimizations. We also demonstrate that some of the basic results of that classical theory of ergodic optimization extend to the C*-dynamical setting.

In Section \ref{Singly generated UE}, we define a value called the \emph{gauge} of a singly generated C*-dynamical system, a non-commutative generalization of the functional of the same name defined in \cite{Assani-Young}, and describe its connections to questions of ergodic optimization, as well as the ways in which it can be used to ``detect" the unique ergodicity of C*-dynamical systems under certain Choquet-theoretic assumptions.

In Section \ref{Amenable UE}, we extend the results of the previous section to the case where the phase group is a countable discrete amenable group. We also provide a characterization of uniquely ergodic C*-dynamical systems of countable discrete amenable groups in terms of various notions of convergence of ergodic averages.

In Section \ref{NC Herman section}, we extend some fundamental identities of ergodic optimization to the noncommutative and relative setting. We also relate the convergence properties of certain ergodic averages to relative ergodic optimizations.

Finally, in Section \ref{NSA}, we provide alternate proofs of several results from this article using the toolbox of nonstandard analysis.

\section{Ergodic Optimization in C*-Dynamical Systems}\label{Ergodic optimization}

Given a unital C*-algebra $\mathfrak{A}$, let $\operatorname{Aut}(\mathfrak{A})$ denote the family of all *-automorphisms of $\mathfrak{A}$. We endow $\operatorname{Aut}(\mathfrak{A})$ with the \emph{point-norm topology}, i.e. the topology induced by the pseudometrics
\begin{align*}
	(\Phi, \Psi)	& \mapsto \left\| \Phi(a) - \Psi(a) \right\|	& (a \in \mathfrak{A}) .
\end{align*}
This topology makes $\operatorname{Aut}(\mathfrak{A})$ a topological group \cite[II.5.5.4]{Blackadar}.

We define a \emph{C*-dynamical system} to be a triple $(\mathfrak{A}, G, \Theta)$ consisting of a unital C*-algebra $\mathfrak{A}$, a topological group $G$ (called the \emph{phase group}), and a point-continuous left group action $\Theta : G \to \operatorname{Aut}(\mathfrak{A})$.

\begin{Not}
Let $(\mathfrak{A}, G, \Theta)$ be a C*-dynamical system, and let $F \subseteq G$ be a nonempty finite subset. We define $\operatorname{Avg}_F : \mathfrak{A} \to \mathfrak{A}$ by
$$\operatorname{Avg}_F x : = \frac{1}{|F|} \sum_{g \in F} \Theta_g a .$$
\end{Not}

Denote by $\mathcal{S}$ the family of all states on $\mathfrak{A}$ endowed with the weak*-topology, and by $\mathcal{T}$ the subfamily of all tracial states on $\mathfrak{A}$. A state $\phi$ on $\mathfrak{A}$ is called \emph{$\Theta$-invariant} (or simply \emph{invariant} if the action $\Theta$ is understood in context) if $\phi = \phi \circ \Theta_g$ for all $g \in G$. Denote by $\mathcal{S}^G \subseteq \mathcal{S}$ the family of all $\Theta$-invariant states on $\mathfrak{A}$, and by $\mathcal{T}^G \subseteq \mathcal{T}$ the family of all $\Theta$-invariant tracial states on $\mathfrak{A}$. The set $\mathcal{S}^G$ (resp. $\mathcal{T}^G$) is weak*-compact in $\mathcal{S}$ (resp. in $\mathcal{T}$). Unless otherwise stated, whenever we deal with subspaces of $\mathcal{S}$, we consider these subspaces equipped with the weak*-topology.

We will assume for the remainder of this section that $(\mathfrak{A}, G, \Theta)$ is a C*-dynamical system such that $\mathfrak{A}$ is separable, and also that $\mathcal{S}^G \neq \emptyset$. This framework will include every system of the form $(C(Y), G, \Theta)$, where $Y$ is a compact metrizable topological space, the group $G$ is countable, discrete, and amenable, and $\Theta$ is of the form $\Theta_g : f \mapsto f \circ U_g$ for all $g \in G$, where $U : G \curvearrowright Y$ is a \emph{right} action of $G$ on $Y$ by homeomorphisms. Because of the correspondence between topological dynamical systems as we've defined them previously in Section \ref{Topological stuff} and C*-dynamical systems over commutative C*-algebras, it is customary to call a C*-dynamical system a ``non-commutative topological dynamical systems."

Before proceeding, we prove the following Krylov–Bogolyubov-type result, which will be useful to establish the $\Theta$-invariance of certain states later.

\begin{Lem}\label{K-B}
Let $(\mathfrak{A}, G, \Theta)$ be a C*-dynamical system, and let $G$ be an amenable group. If $(\phi_k)_{k = 1}^\infty$ is a sequence in $\mathcal{S}$, and $\mathbf{F} = (F_k)_{k = 1}^\infty$ is a right Følner sequence for $G$, then any weak*-limit point of the sequence $\left( \phi_k \circ \operatorname{Avg}_{F_k} \right)_{k = 1}^\infty$ is $\Theta$-invariant. In particular, if $K$ is a nonempty, $\Theta$-invariant, weak*-compact, convex subset of $\mathcal{S}$, then $K \cap \mathcal{S}^G \neq \emptyset$.
\end{Lem}

\begin{proof}
Let $(\phi_k)_{k = 1}^\infty$ be a sequence of states, and fix $g_0 \in G, x \in \mathfrak{A}$. Then
\begin{align*}
	& \left| \phi_k \left( \operatorname{Avg}_{F_k} \Theta_{g_0} x \right) - \phi_k \left( \operatorname{Avg}_{F_k} x \right) \right| \\
	=	& \left| \frac{1}{|F_k|} \left[ \phi_k \left( \sum_{g \in F_k g_0} \Theta_g x \right) - \phi_k \left( \sum_{g \in F_k} \Theta_g x \right) \right] \right| \\
	\leq	& \left| \frac{1}{|F_k|} \phi_k \left( \sum_{g \in F_k g_0 \setminus F_k} \Theta_g x \right) \right| + \left| \frac{1}{|F_k|} \phi_k \left( \sum_{g \in F_k \setminus F_k g_0} \Theta_g x \right) \right| \\
	\leq	& \frac{|F_k g_0 \Delta F_k|}{|F_k|} \|x\| \\
	\stackrel{k \to \infty}{\to}	& 0 .
\end{align*}
Therefore, if $k_1 < k_2 < \cdots$ is such that $\psi = \lim_{\ell \to \infty} \phi_{k_\ell} \circ \operatorname{Avg}_{F_{k_\ell}}$ exists, then
$$\left| \psi(\Theta_g x) - \psi(x) \right|	\leq \limsup_{\ell \to \infty} \frac{ \left|F_{k_\ell} g_0 \Delta F_{k_\ell} \right|}{\left|F_{k_\ell} \right|} \|x\| = 0 .$$

Finally, let $K$ be a nonempty, $\Theta$-invariant, weak*-compact, convex subset of $\mathcal{S}$. Let $\phi$ be any state in $K$, and consider the sequence $\left( \phi \circ \operatorname{Avg}_{F_k} \right)_{k = 1}^\infty$. By the convexity and $\Theta$-invariance of $K$, every term of this sequence is an element of $K$, and since $K$ is compact, there exists a subsequence of this sequence which converges in $K$. As has already been shown, that limit must be an element of $\mathcal{S}^G$.
\end{proof}

\begin{Rmk}
Lemma \ref{Nonstandard K-B} can be seen as a nonstandard-analytic analogue to Lemma \ref{K-B}.
\end{Rmk}

Although our manner of proof of Lemma \ref{K-B} is scarcely novel, the result as we have stated it here can be used to ensure the existence of invariant states with specific properties that might interest us, as seen for example in Corollary \ref{Tracial K-B} and Proposition \ref{Nonempty annihilators}. Our standing hypothesis that $\mathfrak{A}$ be separable is not necessary for this proof of Lemma \ref{K-B}.

\begin{Cor}\label{Tracial K-B}
If $\mathcal{T} \neq \emptyset$, and $G$ is amenable, then $\mathcal{T}^G \neq \emptyset$.
\end{Cor}

\begin{proof}
Apply Lemma \ref{K-B} to the case where $K = \mathcal{T}$.
\end{proof}

\begin{Def}
We denote by $\mathfrak{R}$ the real Banach space of all self-adjoint elements of $\mathfrak{A}$, and denote by $\mathfrak{R}^\natural$ the space of all real self-adjoint bounded linear functionals on $\mathfrak{A}$.
\end{Def}

\begin{Def}\label{Simplex definition}
	Let $V$ be a locally convex topological real vector space, and let $K$ be a compact subset of $V$ which is contained in a hyperplane that does not contain the origin. We call $K$ a \emph{simplex} if the positive cone $P = \left\{ c k : c \in \mathbb{R}_{\geq 0}, k \in K \right\}$ defines a lattice ordering on $P - P = \{ p_1 - p_2 : p_1, p_2 \in P \} \subseteq V$ with respect to the partial order $a \leq b \iff b - a \in P$.
\end{Def}

\begin{Rmk}
	In Definition \ref{Simplex definition}, the assumption that $K$ lives in a hyperplane that does not contain the origin is technically superfluous, but simplifies the theory somewhat (see \cite[Section 10]{Phelps}), and is satisfied by all the simplices that interest us here. Specifically, we know that $\mathcal{S}$ (and by extension $\mathcal{S}^G, \mathcal{T}, \mathcal{T}^G$) lives in the real hyperplane $\left\{ \phi \in \mathfrak{R}^\natural : \phi(1) = 1 \right\}$ defined by the evaluation at $1$.
\end{Rmk}

We begin with the following lemma.

\begin{Lem}\label{Simplices}
	\begin{enumerate}[label=(\roman*)]
		\item The spaces $\mathcal{S}, \mathcal{S}^G, \mathcal{T}, \mathcal{T}^G$ are compact and metrizable.
		\item If $\mathcal{T} \neq \emptyset$, then the space $\mathcal{T}^G$ is a simplex.
	\end{enumerate}
\end{Lem}

Before proving this lemma, we need to introduce some terminology. Let $\phi, \psi$ be two positive linear functionals on a unital C*-algebra $\mathfrak{A}$. We say that the two positive functionals are \emph{orthogonal}, notated $\phi \perp \psi$, if they satisfy either of the following two equivalent conditions:
\begin{enumerate}[label=(\alph*)]
	\item $\| \phi + \psi \| = \| \phi \| + \| \psi \|$.
	\item For every $\epsilon > 0$ exists positive $z \in \mathfrak{A}$ of norm $\leq 1$ such that $\phi(1 - z) < \epsilon, \psi(z) < \epsilon$.
\end{enumerate}
It is well-know that these conditions are equivalent \cite[Lemma 3.2.3]{Pedersen}. For every $\phi \in \mathfrak{R}^\natural$, there exist unique positive linear functionals $\phi^+, \phi^-$ such that $\phi = \phi^+ - \phi^-$, and $\phi^+ \perp \phi^-$, called the \emph{Jordan decomposition} of $\phi$ \cite[II.6.3.4]{Blackadar}.

Before proving Lemma \ref{Simplices}, we demonstrate the following property of the Jordan decomposition of a tracial functional.

\begin{Lem}\label{Tracial Jordan}
	Let $\mathfrak{A}$ be a unital C*-algebra, and $\phi \in \mathfrak{R}^\natural$. Suppose that $\phi(x y) = \phi(y x)$ for all $x, y \in \mathfrak{A}$. Then $\phi^\pm(xy) = \phi^\pm(yx)$ for all $x, y \in \mathfrak{A}$.
\end{Lem}

\begin{proof}
	Let $\mathcal{U}(\mathfrak{A})$ denote the group of unitary elements in $\mathfrak{A}$. For a unitary element $u \in \mathcal{U}(\mathfrak{A})$, let $\operatorname{Ad}_u \in \operatorname{Aut}(\mathfrak{A})$ denote the inner automorphism
	$$\operatorname{Ad}_u x = u x u^* .$$
	Let $\psi \in \mathfrak{A}'$. We claim that $\psi$ is tracial if and only if $\psi \circ \operatorname{Ad}_u = \psi$ for all unitaries $u \in \mathcal{U}(\mathfrak{A})$.
	
	Let $u \in \mathcal{U}(\mathfrak{A})$ be unitary, and $x \in \mathfrak{A}$ an arbitrary element. Then
	\begin{align*}
		\phi(u x)	& = \psi \left(u (x u) u^* \right) \\
		& = \psi(\operatorname{Ad}_u (xu) ) .
	\end{align*}
	So $\psi(ux) = \psi(xu)$ if and only if $\psi(\operatorname{Ad}_u(xu)) = \psi(xu)$.
	
	In one direction, suppose that $\psi = \psi \circ \operatorname{Ad}_u$ for all $u \in \mathcal{U}(\mathfrak{A}).$ Fix $x, y \in \mathfrak{A}$. Then we can write $y = \sum_{j = 1}^4 c_j u_j$ for some $c_1, \ldots, c_4 \in \mathbb{C}$ and unitaries $u_1, \ldots, u_4 \in \mathcal{U} (\mathfrak{A})$ unitary. Then
	\begin{align*}
		\psi(xy)	& = \psi \left( x \sum_{j = 1}^4 c_j u_j \right) \\
		& = \sum_{j = 1}^4 c_j \psi(x u_j) \\
		& = \sum_{j = 1}^4 c_j \psi(\operatorname{Ad}_{u_j} (x u_j)) \\
		& = \sum_{j = 1}^4 c_j \psi(u_j x) \\
		& = \psi \left( \left( \sum_{j = 1}^4 c_j u_j \right) x \right) \\
		& = \psi(y x).
	\end{align*}
	Thus $\psi$ is tracial.
	
	In the other direction, suppose there exists $u \in \mathcal{U}(\mathfrak{A})$ such that $\psi \circ \operatorname{Ad}_u \neq \psi$. Let $y \in \mathfrak{A}$ such that $\psi(y) \neq \psi (\operatorname{Ad}_u y)$, and let $x = y u^*$. Then
	\begin{align*}
		\psi(x u)	& = \psi(y) \\
		& \neq \psi(\operatorname{Ad}_u y) \\
		& = \psi\left( u y u^* \right) \\
		& = \psi(u x) .
	\end{align*}
	Therefore $\psi$ is not tracial.
	
	Now, if $\phi \in \mathfrak{R}^\natural$ is tracial, then $\phi \circ \operatorname{Ad}_u = \phi$ for all $u \in \mathcal{U}(\mathfrak{A})$. Then $\phi = \phi \circ \operatorname{Ad}_u = \left(\phi^+ \circ \operatorname{Ad}_u \right) - \left(\phi^- \circ \operatorname{Ad}_u\right)$. But $\left\| \phi^\pm \circ \operatorname{Ad}_u \right\| = \left\| \phi^\pm \right\|$, so it follows that $\left\| \phi \right\| = \left\| \phi^+ \circ \operatorname{Ad}_u \right\| + \left\| \phi^- \circ \operatorname{Ad}_u \right\|$. Therefore $\phi = \left(\phi^+ \circ \operatorname{Ad}_u\right) - \left(\phi^- \circ \operatorname{Ad}_u\right)$ is an orthogonal decomposition of $\phi$, and so it is \emph{the} Jordan decomposition. This means that $\phi^\pm = \phi^\pm \circ \operatorname{Ad}_u$. Since this is true for all $u \in \mathcal{U}(\mathfrak{A})$, it follows that $\phi^\pm$ are tracial.
\end{proof}

\begin{proof}[Proof of Lemma \ref{Simplices}]
	\begin{enumerate}[label=(\roman*)]
		\item This all follows because $\mathcal{S}$ is a weak*-closed real subspace of the unit ball in the continuous dual of the separable Banach space $\mathfrak{R}$, and the spaces $\mathcal{S}^G, \mathcal{T}, \mathcal{T}^G$ are all closed subspaces of $\mathcal{S}$.
		
		\item It is a standard fact that if $\mathcal{T} \neq \emptyset$, then $\mathcal{T}$ is a simplex \cite[II.6.8.11]{Blackadar}. Let
		$$C^G = \left\{ c \phi : c \in \mathbb{R}_{\geq 0} , \phi \in \mathcal{T}^G \right\}$$ be the positive cone of $\mathcal{T}^G$, and let $\mathfrak{R}^\natural$ denote the (real) space of all bounded self-adjoint tracial linear functionals on $\mathfrak{A}$. Let $E^G$ denote the (real) space of all bounded self-adjoint $\Theta$-invariant linear functionals on $\mathfrak{A}$. We already know that $\mathcal{T}$ lives in a hyperplane of $\mathfrak{R}^\natural$ defined by the evaluation functional $\phi \mapsto \phi(1)$. It will therefore suffice to show that $E^G = C^G - C^G$, and that $E^G$ is a sub-lattice of $\mathfrak{R}^\natural$.
		
		Let $\phi^+, \phi^- \geq 0$ be positive functionals on $\mathfrak{A}$ such that $\phi = \phi^+ - \phi^-$ is tracial, and $\phi^+ \perp \phi^-$. By Lemma \ref{Tracial Jordan}, we know that $\phi^+, \phi^-$ are tracial. We claim that if $\phi \in E^G$, then $\phi^+, \phi^- \in C^G$. To prove this, let $g \in G$, and consider that $\phi^+ \circ \Theta_g, \phi^- \circ \Theta_g$ are both positive linear functionals such that $\phi = \left( \phi^+ \circ \Theta_g \right) - \left( \phi^- \circ \Theta_g \right)$.
		
		We claim that $\left( \phi^+ \circ \Theta_g \right) \perp \left( \phi^- \circ \Theta_g \right)$. Fix $\epsilon > 0$. We know that there exists $z \in \mathfrak{A}$ such that $\| z \| \leq 1, 0 \leq z$, and such that $\phi^+ \left( 1 - z \right) < \epsilon, \phi^- (z) < \epsilon$. Then $\Theta_{g^{-1}} (z)$ is a positive element of norm $\leq 1$ such that
		\begin{align*}
			\phi^+ \left( \Theta_g \left( \Theta_{g^{-1}} (1 - z) \right) \right)	& = \phi^+ (1 - z)	& < \epsilon, \\
			\phi^- \left( \Theta_g \left( \Theta_{g^{-1}} (z) \right) \right)	& = \phi^{-} (z)	& < \epsilon .
		\end{align*}
		Therefore $\left( \phi^+ \circ \Theta_g \right) - \left( \phi^- \circ \Theta_g \right)$ is a Jordan decomposition of $\phi$, and since the Jordan decomposition is unique, it follows that $\phi^+ = \phi^+ \circ \Theta_g, \phi^- = \phi^- \circ \Theta_g$, i.e. that $\phi^+, \phi^- \in C^G$. This means that $E^G = C^G - C^G$.
		
		We now want to show that $E^G = C^G - C^G$ is a sublattice of $E$, i.e. that it is closed under the lattice operations. Let $\phi, \psi \in E^G$. For this calculation, we draw on the identities listed in \cite[Theorem 1.3]{PositiveOperators}. Then
		\begin{align*}
			\phi \lor \psi	& = \left( \left( (\phi - \psi) + \psi \right) \lor (0 + \psi) \right) \\
			& = \left( (\phi - \psi) \lor 0 \right) + \psi \\
			& = (\phi - \psi)^+ + \psi , \\
			\phi \land \psi	& = \left( (\phi - \psi) + \psi \right) \land (0 + \psi) \\
			& = \left( (\phi - \psi) \land 0 \right) + \psi \\
			& = - \left( (- (\phi - \psi)) \lor 0 \right) + \psi \\
			& = - (\psi - \phi)^+ + \psi .
		\end{align*}
		Therefore, if $E^G$ is a real linear space and is closed under the operations $\phi \mapsto \phi^+, \phi \mapsto \phi^-$, then it is also closed under the lattice operations. Thus $E^G$ is a sublattice of $\mathfrak{R}^\natural$.
		
		Hence, the subset $\mathcal{T}^G$ is a compact metrizable simplex.
	\end{enumerate}
\end{proof}

In order to keep our treatment relatively self-contained, we define here several elementary concepts from Choquet theory that will be relevant in this section.

\begin{Def}
	Let $S_1, S_2$ be convex spaces. We call a map $T : S_1 \to S_2$ an \emph{affine map} if for every $v, w \in S_1 ; \; t \in [0, 1]$, we have
	$$T(t v + (1 - t) w) = t T(v) + (1 - t) T(w) .$$
	In the case where $S_2 \subseteq \mathbb{R}$, we call $T$ an \emph{affine functional}.
\end{Def}

\begin{Def}
Let $K$ be a convex subset of a locally convex real topological vector space $V$.
	\begin{enumerate}[label=(\alph*)]
		\item A point $k \in K$ is called an \emph{extreme point} of $K$ if for every pair of points $k_1, k_2 \in K$ and parameter $t \in [0, 1]$ such that $k = t k_1 + (1 - t) k_2$, either $k_1 = k_2$ or $t \in \{0, 1\}$. In other words, we call $k$ extreme if there is no nontrivial way of expressing $k$ as a convex combination of elements of $K$.
		\item The set of all extreme points of $K$ is denoted $\partial_e K$.
		\item A subset $F$ of $K$ is called a \emph{face} if for every pair $k_1, k_2 \in K, t \in (0, 1)$ such that $t k_1 + (1 - t) k_2 \in F$, we have that $k_1, k_2 \in F$.
		\item A face $F$ of $K$ is called an \emph{exposed face} of $K$ if there exists a continuous affine functional $\ell : K \to \mathbb{R}$ such that $\ell(x) = 0$ for all $x \in F$, and $\ell(y) < 0$ for all $y \in K \setminus F$.
		\item A point $k \in K$ is called an \emph{exposed point} of $K$ if $\{k\}$ is an exposed face of $K$.
		\item Given a subset $\mathcal{E}$ of $K$, the closed convex hull of $\mathcal{E}$ is written as $\overline{\operatorname{co}}(\mathcal{E})$.
	\end{enumerate}
\end{Def}

We now introduce the basic concepts in our treatment of ergodic optimization.

\begin{Def}
	Let $x \in \mathfrak{R}$ be a self-adjoint element, and let $K \subseteq \mathcal{S}^G$ be a compact convex subset of $\mathcal{S}^G$. Define a value $m \left( x \vert K \right)$ by
	$$m \left( x \vert K \right) : = \sup_{\psi \in K} \psi(x) .$$
	We say a state $\phi \in K$ is \emph{$(x \vert K)$-maximizing} if $\phi(x) = m (x \vert K)$. Let $K_\mathrm{max}(x) \subseteq K$ denote the set of all $(x \vert K)$-maximizing states. A state $\phi \in K$ is called \emph{uniquely $(x \vert K)$-maximizing} if $K_\mathrm{max}(x) = \{\phi\}$.
\end{Def}

\begin{Rmk}
	We note here a trivial inequality: If $K_1 \subseteq K_2$ are compact convex subsets of $\mathcal{S}^G$, then $m \left( x \vert K_1 \right) \leq m \left( x \vert K_2 \right)$, and in particular, we will always have $m \left( x \vert K_1 \right) \leq m \left( x \vert \mathcal{S}^G \right)$.
\end{Rmk}

We will single out one type of compact convex subset of $\mathcal{S}^G$ which will prove important later. Given a subset $A \subseteq \mathfrak{A}$, set $$\operatorname{Ann}(A) : = \left\{ \phi \in \mathcal{S}^G : A \subseteq \ker \phi \right\} .$$

When $\mathfrak{I} \subseteq \mathfrak{A}$ is a $\Theta$-invariant closed ideal of $\mathfrak{A}$, we have a bijective correspondence between the states in $\operatorname{Ann}(\mathfrak{I})$ and the states on $\mathfrak{A} / \mathfrak{I}$ invariant under the action induced by $\Theta$. We will be referring to this set again in Sections \ref{Singly generated UE} and \ref{Amenable UE}, when values of the form $m \left( a \vert \operatorname{Ann}(A) \right)$ come up in reference to certain ergodic averages. We observe that $\operatorname{Ann}(\{0\}) = \mathcal{S}^G$, and that $A \subseteq B \subseteq \mathfrak{A} \Rightarrow \operatorname{Ann}(A) \supseteq \operatorname{Ann}(B)$. There is also no a priori guarantee that $\operatorname{Ann}(A) \neq \emptyset$, since for example $\operatorname{Ann}(\{1\}) = \emptyset$. However, Proposition \ref{Nonempty annihilators} gives sufficient conditions for $\operatorname{Ann}(A)$ to be nonempty.

\begin{Prop}\label{Nonempty annihilators}
	Let $A \subseteq \mathfrak{A}$ be such that $\Theta_g A \subseteq A$ for all $g \in G$. Suppose there exists a state on $\mathfrak{A}$ which vanishes on $A$. Then $\operatorname{Ann}(A) \neq \emptyset$. In particular, if $\mathfrak{I} \subsetneq \mathfrak{A}$ is a proper closed two-sided ideal of $\mathfrak{A}$ for which $\Theta_g \mathfrak{I} = \mathfrak{I}$ for all $g \in G$, then $\operatorname{Ann}(\mathfrak{I}) \neq \emptyset$.
\end{Prop}

\begin{proof}
	Let $K \subseteq \mathcal{S}$ denote the family of all (not necessarily invariant) states on $\mathfrak{A}$ which vanish on $A$. Then if $\phi \in K$ and $a \in A$, then $\Theta_g a \in A$, so $\phi \circ \Theta_g$ vanishes on $A$. Therefore $\Theta_g K \subseteq K$ for all $g \in G$. It follows from Lemma \ref{K-B} that $K \cap \mathcal{S}^G = \operatorname{Ann}(A) \neq \emptyset$.
	
	Suppose $\mathfrak{I} \subsetneq \mathfrak{A}$ is a proper closed two-sided ideal of $\mathfrak{A}$ for which $\Theta_g \mathfrak{I} = \mathfrak{I}$ for all $g \in G$, and let $\pi : \mathfrak{A} \twoheadrightarrow \mathfrak{A} / \mathfrak{I}$ be the canonical quotient map. Let $\tilde{\Theta} : G \to \operatorname{Aut}(\mathfrak{A} / \mathfrak{I})$ be the induced action of $G$ on $\mathfrak{A} / \mathfrak{I}$ by $\tilde{\Theta}_g(a + \mathfrak{I}) = \Theta_g a + \mathfrak{I}$. Let $\psi$ be a $\tilde{\Theta}$-invariant state on $\mathfrak{A} / \mathfrak{I}$. Then $\psi \circ \pi$ is a $\Theta$-invariant state on $\mathfrak{A}$ which vanishes on $\mathfrak{I}$, i.e. $\psi \circ \pi \in \operatorname{Ann}(\mathfrak{I})$.
\end{proof}

\begin{Prop}\label{Nonempty maximizing states}
	Let $K \subseteq \mathcal{S}^G$ be a nonempty compact convex subset of $\mathcal{S}^G$, and let $x \in \mathfrak{R}$. Then $K_\mathrm{max}(x)$ is a nonempty, compact, exposed face of $K$.
\end{Prop}

\begin{proof}
	To see that $K_\mathrm{max}(x)$ is nonempty, for each $n \in \mathbb{N}$, let $\phi_n \in K$ such that $\phi_n(x) \geq m(x \vert K) - \frac{1}{n}$. Then since $K$ is compact, the sequence $(\phi_n)_{n = 1}^\infty$ has a convergent subsequence. Let $\phi$ be the limit of a convergent subsequence of $(\phi_n)_{n = 1}^\infty$. Then $\phi$ is $(x \vert K)$-maximizing.
	
	To see that $K_\mathrm{max}(x)$ is compact, consider that
	$$K_\mathrm{max}(x) = \left\{ \phi \in K : \phi(x) = m (x \vert K) \right\} ,$$
	which is a closed subset of $K$. As for being an exposed face, consider the continuous affine functional $\ell : K \to \mathbb{R}$ given by
	$$\ell(\phi) = \phi(x) - m(x \vert K) .$$
	Then the functional $\ell$ exposes $K_\mathrm{max}(x \vert K)$, since it is nonpositive on all of $K$ and vanishes exactly on $K_\mathrm{max}(x)$.
\end{proof}

The following result describes the ways in which some ergodic optimizations interact with equivariant *-homomorphisms of C*-dynamical systems.

\begin{Thm}\label{Ergodic optimization through *-homomorphisms}
	Let $\left( \mathfrak{A} , G , \Theta \right) , \left( \tilde{\mathfrak{A}} , G , \tilde{\Theta} \right)$ be two C*-dynamical systems, and let $\pi : \mathfrak{A} \to \tilde{\mathfrak{A}}$ be a surjective *-homomorphism such that
	\begin{align*}
		\tilde{\Theta}_g \circ \pi	& = \pi \circ \Theta_g	& (\forall g \in G) .
	\end{align*}
	Let $\tilde{\mathcal{S}}^G$ denote the space of $\tilde{\Theta}$-invariant states on $\tilde{\Theta}$. Then $m \left( \pi(a) \vert \tilde{\mathcal{S}}^G \right) = m \left( a \vert \operatorname{Ann}(\ker \pi) \right)$.
\end{Thm}

\begin{proof}
	Let $\tilde{\mathcal{S}}^G$ denote the space of $\tilde{\Theta}$-invariant states on $\tilde{\mathfrak{A}}$. We claim that there is a natural bijective correspondence between $\tilde{\mathcal{S}}^G$ and $\operatorname{Ann}(\ker \pi)$. If $\phi$ is a $\tilde{\Theta}$-invariant state on $\tilde{\mathfrak{A}}$, then we can pull it back to a $\Theta$-invariant state $\phi_0$ on $\mathfrak{A}$ by
	$$\phi_0 = \phi \circ \pi .$$
	This $\phi_0$ obviously vanishes on $\ker \pi$, and is $\Theta$-invariant by virtue of the equivariance property of $\pi$. Conversely, if we start with a $\Theta$-invariant state $\psi$ on $\mathfrak{A}$ that vanishes on $\ker \pi$, then we can push it to a $\tilde{\Theta}$-invariant state $\tilde{\psi}$ on $\tilde{ \mathfrak{A} }$ by
	$$\tilde{\psi} \circ \pi = \psi .$$
	
	We claim now that
	$$m \left( a \vert \operatorname{Ann}(\ker \pi) \right) = m \left( \pi(a) \vert \tilde{\mathcal{S}}^G \right).$$
	Let $\phi$ be a $\left( \pi(a) \vert \tilde{\mathcal{S}}^G \right)$-maximizing state on $\tilde{ \mathfrak{A} }$. Then $\phi \circ \pi \in \operatorname{Ann}(\ker \pi)$, so
	$$m \left( \pi(a) \vert \tilde{\mathcal{S}}^G \right) = \phi(\pi(a)) \leq m \left( a \vert \operatorname{Ann}(\ker \pi) \right) .$$
	On the other hand, if $\psi \in \operatorname{Ann}(\ker \pi)$ is $\left( a \vert \operatorname{Ann}(\ker \pi) \right)$-maximizing, then let $\tilde{\psi}$ be such that $\tilde{\psi} \circ \pi = \psi$. Then $\tilde{\psi} \in \tilde{\mathcal{S}}^G$, so
	$$m \left( a \vert \operatorname{Ann}(\ker \pi) \right) = \psi(a) = \tilde{\psi}(\pi(a)) \leq m \left( a \vert \tilde{\mathcal{S}}^G \right) .$$
\end{proof}

The assumption in Theorem \ref{Ergodic optimization through *-homomorphisms} that $\pi$ is surjective is actually superfluous, as shown in Corollary \ref{Ergodic optimization through non-surjective *-homomorphisms}. We will later provide a proof of this stronger claim that uses the gauge functional, introduced in the context of actions of $\mathbb{Z}$ in Section \ref{Singly generated UE} and in the context of actions of amenable groups in Section \ref{Amenable UE}.

Moreover, the proof of Theorem \ref{Ergodic optimization through *-homomorphisms} can be extended to establish a correspondence between ergodic optimization over certain compact convex subsets of $\tilde{\mathcal{S}}^G$ and certain compact convex subsets of $\operatorname{Ann}(\ker \pi)$. For example under the same hypotheses, if $\mathcal{T} \neq \emptyset$, then the proof could be modified in a simple manner to establish that $m \left( \pi(a) \vert \tilde{\mathcal{T}}^G \right) = m \left( a \vert \operatorname{Ann}(\ker \pi) \cap \mathcal{T}^G \right)$, where $\tilde{\mathcal{T}}^G$ denotes the $\tilde{\Theta}$-invariant tracial states on $\tilde{\mathfrak{A}}$. In lieu of stating Theorem \ref{Ergodic optimization through *-homomorphisms} in greater generality, we content ourselves to state this special case (which we will use in future sections) and remark that the argument can be generalized further.

The following characterization of exposed faces in compact metrizable simplices will prove useful.

\begin{Lem}\label{Closed faces are exposed}
	Let $K$ be a compact metrizable simplex. Then every closed face of $K$ is exposed.
\end{Lem}

\begin{proof}
	See \cite[Theorem 7.4]{Davies}.
\end{proof}

The theorem we are building to in this section is as follows.

\begin{Thm}\label{Simplices and exposed faces}
	Let $K \subseteq \mathcal{S}^G$ be a compact simplex. Then the closed faces of $K$ are exactly the sets of the form $K_\mathrm{max}(x)$ for some $x \in \mathfrak{R}$.
\end{Thm}

Before we can prove our main theorem of this section, we will need to prove the following result, which gives us a means by which to build an important linear functional.

\begin{Thm}\label{Extension Theorem}
	Let $K \subseteq \mathcal{S}^G$ be a compact simplex, and let $\ell : K \to \mathbb{R}$ be a continuous affine functional. Then there exists a continuous linear functional $\tilde{\ell} : \overline{\operatorname{span}}_\mathbb{R}(K) \to \mathbb{R}$ such that $\tilde{\ell} \vert_{K} = \ell$.
\end{Thm}

To prove this theorem, we break it up into several parts, attaining the extension $\tilde{\ell}$ as the final step of a few subsequent extensions of $\ell$.

\begin{Lem}\label{First extension}
	Let $K \subseteq \mathcal{S}^G$ be a compact metrizable simplex, and let $\ell : K \to \mathbb{R}$ be a continuous affine functional. Let $P = \left\{ c \phi : c \in \mathbb{R}_{\geq 0}, \phi \in K \right\}$. Then there exists a continuous functional $\ell_1 : P \to \mathbb{R}$ satisfying the following conditions for all $f_1, f_2 \in P ; c \in \mathbb{R}_{\geq 0}$:
	\begin{enumerate}[label=(\alph*)]
		\item $\ell_1(c f_1) = c \ell_1(f_1)$,
		\item $\ell_1(f_1 + f_2) = \ell_1(f_1) + \ell_2(f_2)$,
		\item $\ell_1 \vert_K = \ell$.
	\end{enumerate}
\end{Lem}

\begin{proof}
	Note that every nonzero element of $P$ can be expressed uniquely as $c \phi$ for some $c \in \mathbb{R}_{\geq 0} \setminus \{0\} , \phi \in K$. As such we define
	$$\ell_1(c \phi) = \begin{cases}
		c \ell(\phi)	& c > 0 \\
		0	& c = 0
	\end{cases}$$
	It is immediately clear that this $\ell_1$ satisfies conditions (a) and (c), leaving only (b) to check.
	
	Now, suppose that $f_1 = c_1 \phi_1, f_2 = c_2 \phi_2$ for some $\phi_1, \phi_2 \in K ; c_1, c_2 \in \mathbb{R}_{\geq 0}$. Consider first the case where at least one of $c_1, c_2$ are nonzero.
	Then
	\begin{align*}
		f_1 + f_2	& = c_1 \phi_1 + c_2 \phi_2	\\
		& = (c_1 + c_2) \left( \frac{c_1}{c_1 + c_2} \phi_1 + \frac{c_2}{c_1 + c_2} \phi_2 \right) \\
		\Rightarrow \ell_1(f_1 + f_2)	& = \ell_1 \left( c_1 \phi_1 + c_2 \phi_2 \right)	\\
		& = (c_1 + c_2) \ell \left( \frac{c_1}{c_1 + c_2} \phi_1 + \frac{c_2}{c_1 + c_2} \phi_2 \right) \\
		\textrm{[because $\ell$ is affine]}& = (c_1 + c_2) \left( \frac{c_1}{c_1 + c_2} \ell(\phi_1) + \frac{c_2}{c_1 + c_2} \ell(\phi_2) \right) \\
		& = c_1 \ell(\phi_1) + c_2 \ell(\phi_2) \\
		& = \ell_1(c_1 \phi_1) + \ell_1(c_2 \phi_2) \\
		& = \ell_1(f_1) + \ell_1(f_2) .
	\end{align*}
	In the event that $c_1 = c_2 = 0$, then the additivity property attains trivially.
	
	It remains now to show that $\ell_1$ is continuous. We will check continuity at nonzero points in $P$, and then at $0 \in P$. First, consider the case where $c \phi \in P \setminus \{0\}$, and $c \in \mathbb{R}_{\geq 0}, \phi \in K$. Suppose that $(c_n \phi_n)_n$ is a sequence in $P$ converging in the weak*-topology to $c \phi$. We claim that $c_n \to c$ in $\mathbb{R}$, and $\phi_n \to \phi$ in the weak*-topology.
	
	We first observe that $(c_n \phi_n)(1) = c_n$, so $(c_n)_{n}$ converges in $\mathbb{R}_{\geq 0}$ to $c$, meaning in particular that for sufficiently large $n$, we have that $c_n \in \left[ \frac{c}{2} , \frac{3c}{2} \right] $. Now, if $\lambda : \mathfrak{R} \to \mathbb{R}$ is a norm-continuous linear functional, then
	\begin{align*}
		\lambda(\phi_n)	& = \frac{1}{c_n} \lambda(c_n \phi_n) \\
		& \to \frac{1}{c} \lambda(c \phi) \\
		& = \lambda(\phi) .
	\end{align*}
	Therefore $c_n \to c , \phi_n \to \phi$. Thus we can compute
	\begin{align*}
		\left| \ell_1(c \phi) - \ell_1(c_n \phi_n) \right|	& \leq \left| \ell_1(c \phi) - \ell_1(c_n \phi) \right| + \left| \ell_1(c_n \phi) - \ell_1(c_n \phi_n) \right| \\
		& = |c - c_n| \cdot \left| \ell(\phi) \right| + |c_n| \cdot \left| \ell(\phi) - \ell(\phi_n) \right| \\
		& \leq |c - c_n| \left( \sup_{\phi \in K} |\ell(\phi)| \right) + \frac{3c}{2} \left| \ell(\phi) - \ell(\phi_n) \right| \\
		& \stackrel{n \to \infty}{\to} 0 ,
	\end{align*}
	where $\sup_{\phi \in K} |\ell(\phi)|$ must be finite because $K$ is weak*-compact, and $|\ell(\phi) - \ell(\phi_n)| \stackrel{n \to \infty}{\to} 0$ because $\ell$ is weak*-continuous.
	
	Now, suppose that $(c_n \phi_n)_{n = 1}^\infty$ converges to $0$. Then again we have that $c_n \to 0$ by the same argument used above (i.e. $c_n = (c_n \phi_n)(1)$). Therefore
	$$
	\left| \ell_1 (c_n \phi_n) \right| = |c_n| \cdot \left| \ell (\phi_n) \right| \leq |c_n| \left( \sup_{\phi \in K} |\ell(\phi)| \right) \stackrel{n \to \infty}{\to} 0 .
	$$
	
	We can thus conclude that $\ell_1$ is weak*-continuous.
\end{proof}

\begin{Lem}\label{Second extension}
	Let $\ell_1, P$ be as in Lemma \ref{First extension}, and let $V = P - P$. Then there exists a continuous linear functional $\tilde{\ell} : V \to \mathbb{R}$ such that $\tilde{\ell} \vert_{P} = \ell_1$.
\end{Lem}

\begin{proof}
	Define $\tilde{\ell} : V \to \mathbb{R}$ by
	$$\tilde{\ell}(v) = \ell_1 \left( v^+ \right) - \ell_1 \left( v^- \right),$$
	where $v^+, v^-$ are meant in the sense of the lattice structure $V$ possesses by virtue of $K$ being a simplex.
	
	Our first claim is that if $f, g \in P$ such that $v = f - g$, then $\tilde{\ell}(v) = \ell_1(f) - \ell_1(g)$. To see this, we observe that $f + v^- = g + v^+ \in P$. Therefore
	\begin{align*}
		\ell_1 \left( f + v^- \right)	& = \ell_1 \left( g + v^+ \right) \\
		= \ell_1(f) + \ell_1 \left( v^- \right)	& = \ell_1(g) + \ell_1 \left( v^+ \right) \\
		\Rightarrow \ell_1(f) - \ell_1(g)	& = \ell_1 \left( v^+ \right) - \ell_1 \left( v^- \right) \\
		& = \tilde{\ell}(v) .
	\end{align*}
	This makes linearity fairly straightforward to check. First, to confirm additivity, let $v, w \in V$. Then $v + w = \left( v^+ + w^+ \right) - \left( v^- + w^- \right)$, where $v^+ + w^+, v^- + w^- \in P$. Thus
	\begin{align*}
		\tilde{\ell}(v + w)	& = \ell_1 \left( v^+ + w^+ \right) - \ell_1 \left( v^- + w^- \right) \\
		& = \ell_1 \left( v^+ \right) + \ell_1 \left( w^+ \right) - \ell_1 \left( v^- \right) - \ell_1 \left( w^- \right) \\
		& = \ell_1 \left( v^+ \right) - \ell_1 \left( v^- \right) + \ell_1\left( w^+ \right) - \ell_1 \left( w^- \right) \\
		& = \tilde{\ell}(v) + \tilde{\ell}(w) .
	\end{align*}
	To check homogeneity, let $c \in \mathbb{R}$. If $c \geq 0$, then $cv^+, c v^- \in P$, and $c v^+ - c v^- = c v$; on the other hand, if $c \leq 0$, then $- c v^- , - c v^+ \in P$, and $c v = - c v^- + c v^+$. In both cases, homogeneity is straightforward to show. This proves that $\tilde{\ell}$ is linear.
	
	It is also quick to show that $\tilde{\ell} \vert_{P} = \ell_1$, since if $v \in P$, then $v = v^+$, so $\tilde{\ell}(v) = \ell_1 \left( v^+ \right) - 0 = \ell_1(v)$.
	
	It remains now to show that $\tilde{\ell}$ is continuous. By \cite[Theorem 1.18]{RudinFunctional}, it will suffice to show that $\ker \tilde{\ell}$ is weak*-closed. To prove the kernel is closed, let $(v_n)_{n = 1}^\infty$ be a sequence in $\ker \tilde{\ell}$ converging in the weak*-topology to $v \in V$. By the Uniform Boundedness Principle, it follows that $\sup_{n} \| v_n \| < \infty$. By rescaling, we can assume without loss of generality that $\|v_n\| \leq 1$ for all $n \in \mathbb{N}$, and since the unit ball $B \subseteq V$ is weak*-closed by Banach-Alaoglu, we can infer that $\| v \| \leq 1$.
	
	Since the unit ball $B$ is weak*-compact, it follows that the sequences $\left( v_n^+ \right)_{n = 1}^\infty , \left( v_n^- \right)_{n = 1}^\infty$ have convergent subsequences. Let $(n_j)_{j = 1}^\infty$ be a subsequence along which $v_{n_j}^+ \to m_1 \in P, v_{n_j}^- \to m_2 \in P$. Then if $x \in \mathfrak{R}$, then
	\begin{align*}
		v(x)	& = \lim_{n \to \infty} v_n(x) \\
		& = \lim_{n \to \infty} \left( v_n^+(x) - v_n^-(x) \right) \\
		& = \lim_{j \to \infty} \left( v_{n_j}^+(x) - v_{n_j}^-(x) \right) \\
		& = \left( \lim_{j \to \infty} v_{n_j}^+(x) \right) - \left( \lim_{j \to \infty} v_{n_j}^-(x) \right) \\
		& = m_1(x) - m_2(x) .
	\end{align*}
	Therefore $v = m_1 - m_2$, so
	\begin{align*}
		\tilde{\ell}(v)	& = \tilde{\ell} (m_1) - \tilde{\ell} (m_2) \\
		& = \left( \lim_{j \to \infty} \tilde{\ell} \left( v_{n_j}^+ \right) \right) - \left( \lim_{j \to \infty} \tilde{\ell} \left( v_{n_j}^- \right) \right) \\
		& = \lim_{j \to \infty} \left( \tilde{\ell} \left( v_{n_j}^+ \right) - \tilde{\ell} \left( v_{n_j}^- \right) \right) \\
		& = \lim_{j \to \infty} \tilde{\ell} \left( v_{n_j} \right) \\
		& = \lim_{j \to \infty} 0	\\
		& = 0 .
	\end{align*}
	Therefore, we can conclude that $\tilde{\ell}$ is weak*-continuous.
\end{proof}

\begin{proof}[Proof of Theorem \ref{Extension Theorem}]
	This follows from Lemmas \ref{First extension} and \ref{Second extension}.
\end{proof}

\begin{proof}[Proof of Theorem \ref{Simplices and exposed faces}]
	Let $F \subseteq K$ be a closed face of $K$. By Lemma \ref{Closed faces are exposed}, the face $F$ is exposed, so let $\ell : K \to \mathbb{R}$ be a weak*-continuous affine functional such that
	\begin{align*}
		\ell(k)	& = 0	& (\forall k \in F) , \\
		\ell(k)	& < 0	& (\forall k \in K \setminus F) .
	\end{align*}
	Set
	$$V = \left\{ c_1 \phi_1 - c_2 \phi_2 : c_1, c_2 \in \mathbb{R}_{\geq 0} ; \phi_1, \phi_2 \in K \right\} ,$$
	and let $\tilde{\ell} : V \to \mathbb{R}$ be a continuous linear extension of $\ell$ to $V$ whose existence is promised by Theorem \ref{Extension Theorem}. We can then extend $\tilde{\ell} : V \to \mathbb{R}$ to a weak*-continuous linear functional $\ell' : \mathfrak{R}^\natural \to \mathbb{R}$ \cite[Theorem 3.6]{PositiveOperators}. There thus exists some $x \in \mathbb{R}$ such that $\ell'(\phi) = \phi(x)$ for all $\phi \in \mathfrak{R}^\natural$ \cite[Theorem 5.2]{Baggett}. In particular, we have $\ell'(v) = v(x)$ for all $v \in V$. Therefore $F = K_\mathrm{max}(x)$.
	
	The converse is contained in Proposition \ref{Nonempty maximizing states}.
\end{proof}

In particular, we can recover the following corollary.

\begin{Cor}\label{Uniquely maximizing states}
	If $\phi \in \partial_e K$, then there exists $x \in \mathfrak{R}$ such that $\phi$ is uniquely $(x \vert K)$-maximizing, i.e. such that $\{ \phi \} =  K_\mathrm{max}(x)$.
\end{Cor}

\begin{proof}
	The singleton $\{\phi\}$ is a closed face, and by Lemma \ref{Closed faces are exposed} is therefore an exposed face. Apply Theorem \ref{Simplices and exposed faces}.
\end{proof}

We have developed the language of ergodic optimization here in a somewhat atypical way, where we speak not of $x$-maximizing states \emph{simpliciter}, but of a state that is maximizing relative to a compact convex subset $K$ of $\mathcal{S}^G$, especially a compact simplex $K$. This notion of relative ergodic optimization has precedent in \cite{ObservableMeasures}. For our purposes, this relative ergodic optimization means we can consider ergodic optimization problems over different \emph{types} of states. In Section \ref{NC Herman section}, we will broaden our scope somewhat to consider ergodic optimization in the noncommutative setting relative to a set of states that aren't necessarily $\Theta$-invariant.

Since Theorem \ref{Simplices and exposed faces} applies in cases where $K$ is a simplex, we will conclude this section by describing some situations where $\mathcal{S}^G$ is a compact metrizable simplex.

For each $\phi \in \mathcal{S}^G$, let $\pi_\phi : \mathfrak{A} \to \mathscr{B} (\mathscr{H}_\phi)$ be the GNS representation corresponding to $\phi$. Define a unitary representation $u_\phi : G \to \mathbb{U}(\mathscr{H}_\phi)$ of $G$ by
$$u_\phi(g) \pi_\phi(a) = \pi_\phi \left( \Theta_{g^{-1}} (a) \right) ,$$
extending this from $\pi_\phi(\mathfrak{A})$ to $\mathscr{H}_\phi$. Set
$$E_\phi = \left\{ v \in \mathscr{H}_\phi : u_\phi(v) = v \textrm{ for all } g \in G \right\} .$$
Let $P_\phi : \mathscr{H}_\phi \twoheadrightarrow E_\phi$ be the orthogonal projection (in the functional-analytic sense) of $\mathscr{H}_\phi$ onto $E_\phi$. We call the C*-dynamical system $(\mathfrak{A}, G, \Theta)$ a \emph{$G$-abelian} system if for every $\phi \in \mathcal{S}^G$, the family of operators $\left\{ P_\phi \pi_\phi(a) P_\phi \in \mathscr{B}(\mathscr{H}_\phi) : a \in \mathfrak{A} \right\}$ is mutually commutative.

We record here a handful of germane facts about $G$-abelian systems.

\begin{Prop}
	If $(\mathfrak{A}, G, \Theta)$ is $G$-abelian, then $\mathcal{S}^G$ is a simplex.
\end{Prop}

\begin{proof}
	See \cite[Theorem 3.1.14]{Sakai}.
\end{proof}

\begin{Def}
	We call a system $(\mathfrak{A}, G, \Theta)$ \emph{asymptotically abelian} if there exists a sequence $(g_n)_{n = 1}^\infty$ in $G$ such that
	$$\left[ \Theta_{g_n} a, b \right] \stackrel{n \to \infty}{\to} 0 $$
	for all $a, b \in \mathfrak{A}$, where $\left[ \cdot , \cdot \right]$ is the Lie bracket $[x, y] = xy - yx$ on $\mathfrak{A}$.
\end{Def}

\begin{Prop}
	If $(\mathfrak{A}, G, \Theta)$ is asymptotically abelian, then it is also $G$-abelian.
\end{Prop}

\begin{proof}
	See \cite[Proposition 3.1.16]{Sakai}.
\end{proof}

\section{Unique ergodicity and gauges: the singly generated setting}\label{Singly generated UE}

So far we have spoken about C*-dynamical systems, a noncommutative analog of a topological dynamical systems. But just as classical ergodic theory is often interested in the interplay between topological dynamical systems and the measure-theoretic dynamical systems they can be realized in, we are interested in questions about the interplay between C*-dynamical systems and the non-commutative measure-theoretic dynamical systems they can be realized in. To make this more precise, we introduce the notion of a W*-dynamical system.

A \emph{W*-probability space} is a pair $(\mathfrak{M}, \rho)$ consisting of a von Neumann algebra $\mathfrak{M}$ and a faithful tracial normal state $\rho$ on $\mathfrak{M}$. An \emph{automorphism} of a W*-probability space $(\mathfrak{M}, \rho)$ is a *-automorphism $T : \mathfrak{M} \to \mathfrak{M}$ such that $\rho \circ T = \rho$, i.e. an automorphism of $\mathfrak{M}$ which preserves $\rho$. A \emph{W*-dynamical system} is a quadruple $(\mathfrak{M}, \rho, G, \Xi)$, where $(\mathfrak{M}, \rho)$ is a W*-probability space, and $\Xi : G \to \operatorname{Aut}(\mathfrak{M}, \rho)$ is a left action of a discrete topological group $G$ (called the \emph{phase group}) on $\mathfrak{M}$ by $\rho$-preserving automorphisms of $\mathfrak{M}$, i.e. such that $\rho(\Xi_g x) = \rho(x)$ for all $g \in G, x \in \mathfrak{M}$. Importantly, if $(\mathfrak{M}, \rho, G, \Xi)$ is a W*-dynamical system, then $(\mathfrak{M}, G, \Xi)$ is automatically a W*-dynamical system.

\begin{Rmk}
In the literature, the term ``W*-dynamical system" is sometimes used to refer to a more general construction, where the group $G$ is assumed to satisfy some topological conditions, and the action is assumed to be continuous in the strong operator topology, e.g. \cite{NoncommutativeJoinings}. Other authors use a yet more general definition, e.g. \cite[III.3.2]{Blackadar}. Since we are only interested in actions of discrete groups, we adopt a narrower definition.
\end{Rmk}

\begin{Def}
	Given a W*-probability space, we define $\mathcal{L}^2(\mathfrak{M}, \rho)$ to be the Hilbert space defined by completing $\mathfrak{M}$ with respect to the inner product $\left< x , y \right>_\rho = \rho \left( y^* x \right)$, i.e. the Hilbert space associated with the faithful GNS representation of $\mathfrak{M}$ induced by $\rho$.
\end{Def}

Finally, we introduce the notion of a C*-model, intending to generalize the notion of a topological model from classical ergodic theory to this noncommutative setting.

\begin{Def}
	Let $(\mathfrak{M}, \rho, G, \Xi)$ be a W*-dynamical system. A \emph{C*-model} of $(\mathfrak{M}, \rho, G, \Xi)$ is a quadruple $(\mathfrak{A}, G, \Theta; \iota)$ consisting of a C*-dynamical system $(\mathfrak{A}, G, \Theta)$ and a *-homomorphism $\iota : \mathfrak{A} \to \mathfrak{M}$ such that
	\begin{enumerate}[label=(\alph*)]
		\item $\iota(\mathfrak{A})$ is dense in the weak operator topology of $\mathfrak{M}$,
		\item $\Xi_g \left( \iota(\mathfrak{A}) \right) = \iota(\mathfrak{A})$ for all $g \in G$, and
		\item $\Xi_g \circ \iota = \iota \circ \Theta_g$ for all $g \in G$.
	\end{enumerate}
	We call the C*-model $(\mathfrak{A}, G, \Theta; \iota)$ \emph{faithful} if $\iota$ is also injective.
\end{Def}

We remark that we can turn any C*-model into a faithful C*-model through a quotienting process. If $\iota$ was not injective, then we could instead consider $\tilde{\iota} : \mathfrak{A} / \ker \iota \hookrightarrow \mathfrak{M}$. In the case where $\mathfrak{A}$ is commutative, this quotienting process corresponds (via the Gelfand-Naimark Theorem) to taking a measure-theoretic dynamical system and restricting to the support of the resident probability measure. To see this, let $\mathfrak{A} = C(X)$, where $X$ is a compact metrizable topological space, and let $\mathfrak{M} = L^\infty(X, \mu)$ for some Borel probability measure $\mu$. Let $\iota : C(X) \to L^\infty(X, \mu)$ be the (not necessarily injective) map that maps a continuous function on $X$ to its equivalence class in $L^\infty(X, \mu)$. It can be seen that $f \in \ker \iota$ if and only if the open set $\left\{ x \in X : f(x) \neq 0 \right\}$ is of measure $0$, or equivalently if $f \vert_{\operatorname{supp}(\mu)} = 0$, and in particular that $\iota$ is injective if and only if $\mu$ is \emph{strictly positive} (i.e. $\mu$ assigns positive measure to all nonempty open sets). As such, we can identify $C(X) / \ker \iota$ with $C(\operatorname{supp}(\mu))$. Let $Y = \operatorname{supp}(\mu)$ denote the support of $\mu$ on $X$, and let $\pi : C(X) \twoheadrightarrow C(Y)$ be the quotient map (which corresponds to a restriction from $X$ to $Y$, i.e. $\pi f = f \vert_Y$). Then algebraically, we have a commutative diagram
$$
\begin{tikzcd}
	C(X) \arrow[r, two heads, "\pi"] \arrow[rd, "\iota"]	& C(Y) \arrow[d, hook, dotted, "\tilde{\iota}"] \\
	& L^\infty(X, \mu)
\end{tikzcd}
$$
So in the commutative case, we can make $\iota : C(X) \to L^\infty(X, \mu)$ injective by looking at $\tilde{\iota} : C(Y) \to L^\infty(Y, \mu) \cong L^\infty(X, \mu)$, i.e. by using the support $Y$ to model $(Y, \mu) \cong (X, \mu)$.

Importantly, so long as $\mathcal{L}^2(\mathfrak{M}, \rho)$ is separable, any W*-dynamical system $(\mathfrak{M}, \rho, G, \Xi)$ will admit a faithful separable C*-model. To construct such a C*-model, it suffices to take some separable C*-subalgebra $\mathfrak{B} \subseteq \mathfrak{M}$ which is dense in $\mathfrak{M}$ with respect to the weak operator topology, then let $\mathfrak{A}$ be the norm-closure of the span of $\bigcup_{g \in G} \left( \Xi_g \mathfrak{B} \right)$. We then define $\Theta_g = \Xi_g \vert_{\mathfrak{A}}$ and let $\iota : \mathfrak{A} \hookrightarrow \mathfrak{M}$ be the inclusion map.

One last important concept in this section and the next will be unique ergodicity. A C*-dynamical system $(\mathfrak{A}, G, \Theta)$ is called \emph{uniquely ergodic} if $\mathcal{S}^G$ is a singleton. As in the commutative setting, unique ergodicity can be equivalently characterized in terms of convergence properties of ergodic averages. To our knowledge, the strongest such characterization of unique ergodicity for singly generated C*-dynamical systems can be found in \cite[Theorem 3.2]{AbadieDykema}, which describes unique ergodicity relative to the fixed point subalgebra. This characterization was then generalized to characterize unique ergodicity relative to the fixed point subalgebra for C*-dynamical systems over amenable phase groups in \cite[Theorem 5.2]{DuvenhageStroeh}; however, in Corollary \ref{Unique ergodicity for C*-dynamical systems in terms of gauge}, we provide a characterization of uniquely ergodic C*-dynamical systems in terms of ergodic averages that is not encompassed by \cite[Theorem 5.2]{DuvenhageStroeh}.

Given a C*-dynamical system $(\mathfrak{A}, \mathbb{Z}, \Theta)$, let $a \in \mathfrak{A}$ be a positive element. We define the \emph{gauge} of $a$ to be
$$\Gamma(a) : = \lim_{k \to \infty} \frac{1}{k} \left\| \sum_{j = 0}^{k - 1} \Theta_j a \right\| .$$
To prove this limit exists, it suffices to observe that the sequence $\left( \left\|\sum_{j = 0}^{k - 1} \Theta_j a \right\| \right)_{k = 1}^\infty$ is subadditive, since
\begin{align*}
	\left\| \sum_{j = 0}^{k + \ell - 1} \Theta_j a \right\|	& \leq \left\| \sum_{j = 0}^{k - 1} \Theta_j a \right\| + \left\| \sum_{j = k}^{k + \ell - 1} \Theta_j a \right\| \\
	& = \left\| \sum_{j = 0}^{k - 1} \Theta_j a \right\| + \left\| \Theta_k \sum_{j = 0}^{\ell - 1} \Theta_j a \right\| \\
	& = \left\| \sum_{j = 0}^{k - 1} \Theta_j a \right\| + \left\| \sum_{j = 0}^{\ell - 1} \Theta_j a \right\| .
\end{align*}
Therefore, by the Subadditivity Lemma, the sequence $\left( \frac{1}{k} \left\|\sum_{j = 0}^{k - 1} \Theta_j a \right\| \right)_{k = 1}^\infty$ converges, and we have the equality
$$\lim_{k \to \infty} \frac{1}{k} \left\| \sum_{j = 0}^{k - 1} \Theta_j a \right\| = \inf_{k \in \mathbb{N}} \frac{1}{k} \left\| \sum_{j = 0}^{k - 1} \Theta_j a \right\| .$$

We have the following characterization of $\Gamma$ in the language of ergodic optimization.

\begin{Thm}\label{Singly generated gauge}
Let $(\mathfrak{A}, \mathbb{Z}, \Theta)$ be a C*-dynamical system. Then if $a \in \mathfrak{A}$ is a positive element, then $\Gamma(a) = m \left( a \vert \mathcal{S}^G \right)$.
\end{Thm}

\begin{proof}
	For each $k \in \mathbb{N}$, choose a state $\sigma_k$ on $\mathfrak{A}$ such that
	$$\sigma_k \left( \frac{1}{k} \sum_{j = 0}^{k - 1} \Theta_j a \right) = \left\| \frac{1}{k} \sum_{j = 0}^{k - 1} \Theta_j a \right\| .$$
	Let $\omega_k = \frac{1}{k} \sum_{j = 0}^{k - 1} \sigma_k \circ \Theta_j$, so
	\begin{align*}
		\omega_k(x)	& = \frac{1}{k} \sum_{j = 0}^{k - 1} \sigma_k \left( \Theta_j x \right)  \\
		& = \sigma_k \left( \frac{1}{k} \sum_{j = 0}^{k - 1} \Theta_j x \right) , \\
		\omega_k(a)	& = \sigma_k \left( \frac{1}{k} \sum_{j = 0}^{k - 1} \Theta_j a \right) \\
		& = \left\| \frac{1}{k} \sum_{j = 0}^{k - 1} \Theta_j a \right\| .
	\end{align*}
	Let $\omega \in \mathcal{S}$ be a weak*-limit point of $\left( \omega_k : k \in \mathbb{N} \right)$, and let $k_1 < k_2 < \cdots$ be a subsequence such that $\omega_{k_n} \stackrel{n \to \infty}{\to} \omega$ in the weak*-topology. By Lemma \ref{K-B}, we know that $\omega$ is $\Theta$-invariant. Therefore $\omega(a) = \Gamma(a)$, and $\omega$ is a $\Theta$-invariant state on $\mathfrak{A}$, so $$\Gamma(a) = \omega(a) \leq m \left( a \vert \mathcal{S}^\mathbb{Z} \right) .$$
	
	Now, we prove the opposite inequality. Let $\phi \in \mathcal{S}^\mathbb{Z}$. Then
	\begin{align*}
		\phi(a)	& = \phi \left( \operatorname{Avg}_k a \right) \\
		& \leq \left\| \operatorname{Avg}_k a \right\| \\
		& = \frac{1}{k} \left\| \sum_{j = 0}^{k - 1} \Theta_j a \right\| \\
		& = \frac{1}{k} \left\| \sum_{j = 0}^{k - 1} \Theta_j a \right\|	& (\forall k \in \mathbb{N}) \\
		\Rightarrow \phi(a)	& \leq \inf_{k \in \mathbb{N}} \frac{1}{k} \left\| \sum_{j = 0}^{k - 1} \Theta_j a \right\| \\
		& = \Gamma(a) \\
		\Rightarrow \sup_{\psi \in \mathcal{S}^\mathbb{Z}} \psi(a)	& \leq \Gamma(a) .
	\end{align*}
	Therefore
	$$m \left( a \vert \mathcal{S}^\mathbb{Z} \right) = \sup_{\psi \in \mathcal{S}^\mathbb{Z}} \psi(a) \leq \Gamma(a) .$$
	
	This establishes the identity.
\end{proof}

\begin{Cor}\label{Gamma estimate}
	Let $(\mathfrak{M}, \rho, \mathbb{Z}, \Xi)$ be a W*-dynamical system, and let $(\mathfrak{A}, \mathbb{Z}, \Theta; \iota)$ be a C*-model of $(\mathfrak{M}, \rho, \mathbb{Z}, \Xi)$. If $a \in \mathfrak{A}$ is a positive element, then
	$$\Gamma(\iota(a)) = m \left( a \vert \operatorname{Ann}(\ker \iota) \right) .$$
\end{Cor}

\begin{proof}
	Write $\tilde{\mathfrak{A}} = \iota(\mathfrak{A}) \subseteq \mathfrak{M}$, and let $\tilde{\Theta} : \mathbb{Z} \to \operatorname{Aut} \left( \tilde{\mathfrak{A}} \right)$ be the action $\tilde{ \Theta }_n = \Xi_n \vert_{ \tilde{\mathfrak{A}} }$ obtained by restricting $\Xi$ to $\tilde{\mathfrak{A}}$. Write $\tilde{\mathcal{S}}^\mathbb{Z}$ for the space of $\tilde{\Theta}$-invariant states on $\tilde{\mathfrak{A}}$.
	
	We can write $\Gamma_{\mathfrak{M}}(\iota(a)) = \Gamma_{\tilde{\mathfrak{A}}}(\iota(a))$. By Theorem \ref{Singly generated gauge}, we know that $\Gamma_{ \tilde{\mathfrak{A}} }(\iota(a)) = m \left( \iota(a) \vert \tilde{\mathcal{S}}^\mathbb{Z} \right)$, and by Theorem \ref{Ergodic optimization through *-homomorphisms}, we know that $m \left( \iota(a) \vert \tilde{\mathcal{S}}^\mathbb{Z} \right) = m \left( a \vert \operatorname{Ann}(\ker \iota) \right)$.
\end{proof}

\begin{Rmk}
	Corollary \ref{Gamma estimate} can be regarded as an operator-algebraic extension of Lemma 2.3 from \cite{Assani-Young}. The assumption that $(\mathfrak{A}, G, \Theta; \iota)$ is faithful can be understood as analogous to the assumption of strict positivity in that paper.
\end{Rmk}

This $\Gamma$ value provides an alternative characterization of unique ergodicity, at least under some additional Choquet-theoretic hypotheses.

\begin{Thm}
	Let $(\mathfrak{M}, \rho, \mathbb{Z}, \Xi)$ be a W*-dynamical system, and let $(\mathfrak{A}, \mathbb{Z}, \Theta; \iota)$ be a faithful C*-model of $(\mathfrak{M}, \rho, \mathbb{Z}, \Xi)$. Then the following conditions are related by the implications (i)$\iff$(ii)$\Rightarrow$(iii).
	\begin{enumerate}[label=(\roman*)]
		\item The C*-dynamical system $(\mathfrak{A}, \mathbb{Z}, \Theta)$ is uniquely ergodic.
		\item The C*-dynamical system $(\mathfrak{A}, \mathbb{Z}, \Theta)$ is strictly ergodic.
		\item $\Gamma(\iota(a)) = \rho(\iota(a))$ for all positive $a \in \mathfrak{A}$.
	\end{enumerate}
	Further, if $\mathcal{S}^\mathbb{Z}$ is a simplex, then (iii)$\Rightarrow$(i).
\end{Thm}

\begin{proof}
	(i)$\Rightarrow$(ii) Suppose that $(\mathfrak{A}, \mathbb{Z}, \Theta)$ is uniquely ergodic. Then $\rho \circ \iota$ is an invariant state on $\mathfrak{A}$, so it follows that $\rho \circ \iota$ is the unique invariant state on $\mathfrak{A}$. But $\rho \circ \iota$ is also a faithful state on $\mathfrak{A}$, so it follows that $(\mathfrak{A}, \mathbb{Z}, \Theta)$ is strictly ergodic.
	
	(ii)$\Rightarrow$(i) Trivial.
	
	(i)$\Rightarrow$(iii) Suppose that $(\mathfrak{A}, \mathbb{Z}, \Theta)$ is uniquely ergodic, and let $a \in \mathfrak{A}$ be positive. Let $\phi$ be a $\mathcal{S}^\mathbb{Z}$-maximizing state for $a$. Then $\phi = \rho \circ \iota$, since both $\phi$ and $\rho \circ \iota$ are invariant states on $\mathfrak{A}$, and $(\mathfrak{A}, \mathbb{Z}, \Theta)$ is uniquely ergodic. Thus $\phi = \rho \circ \iota$, so $\Gamma(\iota(a)) = \phi(a) = \rho(\iota(a))$.
	
	(iii)$\Rightarrow$(i) Suppose that $\mathcal{S}^\mathbb{Z}$ is a simplex, but that $(\mathfrak{A}, \mathbb{Z}, \Theta)$ is \emph{not} uniquely ergodic. By the Krein-Milman Theorem, there exists two distinct extreme points of $\mathcal{S}^\mathbb{Z}$, and in particular there exists an extreme point $\phi \in \mathcal{S}^\mathbb{Z}$ of $\mathcal{S}^\mathbb{Z}$ distinct from $\rho \circ \iota$. Then by Corollary \ref{Uniquely maximizing states}, there exists $a \in \mathfrak{A}$ self-adjoint such that $\{ \phi \} = \mathcal{S}_\mathrm{max}^\mathbb{Z}(a)$. We can assume that $a$ is positive, since otherwise we could replace $a$ with $a + r$ for a sufficiently large positive real number $r > 0$, and $\mathcal{S}_\mathrm{max}^\mathbb{Z}(a) = \mathcal{S}_\mathrm{max}^\mathbb{Z}(a + r)$. Then $\Gamma(\iota(a)) = \phi(a)$. But by the assumption that $\phi$ is uniquely $\left( a \vert \mathcal{S}^\mathbb{Z} \right)$-maximizing, it follows that $\rho(\iota(a)) < \phi(a)$. Therefore $\Gamma(\iota(a)) \neq \rho(\iota(a))$, meaning that (iii) does not attain. Thus $\neg$(i)$\Rightarrow \neg$(iii).
\end{proof}

\section{Unique ergodicity and gauge: the amenable setting}\label{Amenable UE}

For the duration of this section, we assume that $(\mathfrak{M}, \rho, G, \Xi)$ is a W*-dynamical system with $\mathcal{L}^2(\mathfrak{M}, \rho)$ separable. Assume further that $(\mathfrak{A}, G, \Theta)$ is a C*-dynamical system such that $\mathfrak{A}$ is separable, and that $G$ is amenable. It follows from Corollary \ref{Tracial K-B} that $\mathcal{S}^G \neq \emptyset$.

In this section, we expand upon some of the ideas presented in Section \ref{Singly generated UE}, generalizing from the case of actions of $\mathbb{Z}$ to actions of a countable discrete amenable group $G$. We separate these two sections because our treatment of the more general amenable setting has some additional nuances to it.

Our first result of this section is a generalization of a classical result from ergodic theory regarding unique ergodicity, which is that a (singly generated) topological dynamical system is uniquely ergodic if and only if the averages of the continuous functions converge to a constant. This classical result is well-known, and can be found in many standard texts on ergodic theory, e.g. \cite[Thm 6.2.1]{DajaniDirksin}, \cite[Thm 10.6]{EisnerOperators}, \cite[Thm 5.17]{Walters}, but the earliest example of a result like this that we could find was \cite[5.3]{OxtobyErgodic}. Theorem \ref{Unique ergodicity equivalent statements} generalizes this classical result not only to the noncommutative setting, but to the setting where the phase group $G$ is amenable.

We define the \emph{weak topology} on a C*-algebra $\mathfrak{A}$ to be the topology generated by the states on $\mathfrak{A}$, i.e.
\begin{align*}
	x	& \mapsto \psi(x)	& (\psi \in \mathcal{S}) .
\end{align*}
In other words, the weak topology is the topology in which a net $(x_i)_{i \in \mathscr{I}}$ converges to $x$ if and only if $(\psi(x_i) )_{i \in \mathscr{I}}$ converges to $\psi(x)$ for every state $\psi$ on $\mathfrak{A}$. We say the net $(x_i)_{i \in \mathscr{I}}$ \emph{converges weakly} to $x$ if it converges in the weak topology.

\begin{Thm}\label{Unique ergodicity equivalent statements}
	Let $(\mathfrak{A}, G, \Theta)$ be a C*-dynamical system. Then the following conditions are equivalent.
	\begin{enumerate}[label=(\roman*)]
		\item $(\mathfrak{A}, G, \Theta)$ is uniquely ergodic.
		\item There exists a right Følner sequences $(F_k)_{k = 1}^\infty$ for $G$ and a linear functional $\phi : \mathfrak{A} \to \mathbb{C}$ such that for all $x \in \mathfrak{A}$, the sequence $\left( \operatorname{Avg}_{F_k} x \right)_{k = 1}^\infty$ converges in norm to $\phi(x) 1 \in \mathbb{C} 1$.
		\item There exists a left Følner sequences $(F_k)_{k = 1}^\infty$ for $G$ and a linear functional $\phi : \mathfrak{A} \to \mathbb{C}$ such that for all $x \in \mathfrak{A}$, the sequence $\left( \operatorname{Avg}_{F_k} x \right)_{k = 1}^\infty$ converges weakly to $\phi(x) 1 \in \mathbb{C} 1$.
		\item There exists a state $\phi$ on $\mathfrak{A}$ such that for every right Følner sequence $(F_k)_{k = 1}^\infty$ for $G$, the sequence $\left( \operatorname{Avg}_{F_k} x \right)_{k = 1}^\infty$ converges in norm to $\phi(x) 1 \in \mathbb{C} 1$.
		\item There exists a state $\phi$ on $\mathfrak{A}$ such that for every left Følner sequence $(F_k)_{k = 1}^\infty$ for $G$, the sequence $\left( \operatorname{Avg}_{F_k} x \right)_{k = 1}^\infty$ converges weakly to $\phi(x) 1 \in \mathbb{C} 1$.
	\end{enumerate}
\end{Thm}

\begin{proof}
	Assume throughout that any $x \in \mathfrak{A}$ is nonzero.	
	
	(ii)$\Rightarrow$(iii) Follows from the existence of two-sided Følner sequence.
	
	(iv)$\Rightarrow$(v) Follows from the existence of two-sided Følner sequence.
	
	(iv)$\Rightarrow$(ii) Trivial.
	
	(v)$\Rightarrow$(iii) Trivial.
	
	(iii)$\Rightarrow$(i) Suppose that $\operatorname{Avg}_{F_k} x \to \phi(x) 1 \in \mathbb{C} 1$ weakly for all $x \in \mathfrak{A}$. We claim that $\phi$ is the unique invariant state of $(\mathfrak{A}, G, \Theta)$. First, we demonstrate that $\phi$ is $\Theta$-invariant. Fix $g_0 \in G$, and fix $\epsilon > 0$. Choose $K_1, K_2, K_3 \in \mathbb{N}$ such that
	\begin{align*}
		k	& \geq K_1	& \Rightarrow \left| \phi(\phi(x) 1) - \phi \left( \operatorname{Avg}_{F_k} x \right) \right|	& < \frac{\epsilon}{3} , \\
		k	& \geq K_2	& \Rightarrow \left| \phi(\Theta_{g_0} \phi(x) 1) - \phi (\Theta_{g_0} \operatorname{Avg}_{F_k} x ) \right|	& < \frac{\epsilon}{3} , \\
		k	& \geq K_3	& \Rightarrow \frac{|g_0 F_k \Delta F_k|}{|F_k|}	& < \frac{\epsilon}{3 \| x \|} .
	\end{align*}
	The $K_1, K_2$ exist because we know that in the weak topology, the functionals $\phi , \phi \circ \Theta_{g_0}$ are both continuous, and $K_3$ exists by the amenability of $G$. Let $K = \max \{K_1, K_2, K_3\}$. Then if $k \geq K$, then
	\begin{align*}
		\left| \phi( \Theta_{g_0} x) - \phi(x) \right|	\leq & \left| \phi(\Theta_{g_0} x) - \phi(\Theta_{g_0} \operatorname{Avg}_{F_k} x) \right| \\
		& + \left| \phi(\Theta_{g_0} \operatorname{Avg}_{F_k} x) - \phi(\operatorname{Avg}_{F_k} x) \right| + \left| \phi(\operatorname{Avg}_{F_k} x) - \phi(x) \right| \\
		\leq & \frac{\epsilon}{3} + \left| \phi(\Theta_{g_0} \operatorname{Avg}_{F_k} x) - \phi(\operatorname{Avg}_{F_k} x) \right| + \frac{\epsilon}{3} \\
		= & \frac{2 \epsilon}{3} + \left| \phi(\Theta_{g_0} \operatorname{Avg}_{F_k} x) - \phi(\operatorname{Avg}_{F_k} x) \right| \\
		= & \frac{2 \epsilon}{3} + \left| \phi \left( \frac{1}{|F_k|}\left( \sum_{g \in F_k} \Theta_{g_0 g} x \right) - \left( \frac{1}{|F_k|} \sum_{g \in F_k} \Theta_g x \right) \right) \right| \\
		= & \frac{2 \epsilon}{3} + \left| \phi \left( \frac{1}{|F_k|}\left( \sum_{g \in g_0 F_k} \Theta_{g} x \right) - \left( \frac{1}{|F_k|} \sum_{g \in F_k} \Theta_g x \right) \right) \right| \\
		= & \frac{2 \epsilon}{3} + \left| \phi \left( \frac{1}{|F_k|}\left( \sum_{g \in g_0 F_k \setminus F_k} \Theta_{g} x \right) - \left( \frac{1}{|F_k|} \sum_{g \in F_k \setminus g_0 F_k} \Theta_g x \right) \right) \right| \\
		\leq	& \frac{2 \epsilon}{3} + \left| \phi \left( \frac{1}{|F_k|}\sum_{g \in g_0 F_k \setminus F_k} \Theta_{g} x \right) \right| + \left| \phi \left( \frac{1}{|F_k|} \sum_{g \in F_k \setminus g_0 F_k} \Theta_g x \right) \right| \\
		<	& \frac{2 \epsilon}{3} + \frac{|g_0 F_k \Delta F_k|}{|F_k|} \| x \| \\
		= & \epsilon .
	\end{align*}
	Therefore $\phi$ is $\Theta$-invariant. To see that it is positive, it suffices to observe that $x \geq 0 \Rightarrow \operatorname{Avg}_{F_k} x \geq 0$, meaning that $\phi(x) = \lim_{k \to \infty} \phi(\operatorname{Avg}_{F_k} x) \geq 0$. To see that $\phi(1) = 1$, we just observe that $\operatorname{Avg}_{F_k} 1 = 1$ for all $k \in \mathbb{N}$.
	
	Now we show that $\phi$ is the \emph{unique} $\Theta$-invariant state. Let $\psi$ be any invariant state. Then
	\begin{align*}
		\psi(x)	& = \psi(\operatorname{Avg}_{F_k} x) \\
		& \stackrel{k \to \infty}{\to} \psi(\phi(x) 1) \\
		& = \phi(x) \psi(1) \\
		& = \phi(x) .
	\end{align*}
	Therefore $\psi = \phi$, and so $(\mathfrak{A}, G, \Theta)$ is uniquely ergodic.
	
	(i)$\Rightarrow$(iv) Fix a right Følner sequence $(F_k)_{k = 1}^\infty$, and assume for contradiction that $(\mathfrak{A}, G, \Theta)$ is uniquely ergodic with $\Theta$-invariant state $\phi$, but that there exists $x \in \mathfrak{A}$ such that $\left( \operatorname{Avg}_{F_k} x \right)_{k = 1}^\infty$ does \emph{not} converge in norm to a scalar, and in particular does not converge in norm to $\phi(x) 1$. Since we can decompose $x$ into its real and imaginary parts, we can assume that $x \in \mathfrak{A}_\mathrm{sa}$. Fix $\epsilon_0 > 0$ for which there exists an infinite sequence $k_1 < k_2 < \cdots$ such that $\left\| \operatorname{Avg}_{F_{k_n}} x - \phi(x) 1 \right\| \geq \epsilon_0$. Then for each $n \in \mathbb{N}$ exists a state $\psi_n$ on $\mathfrak{A}$ such that $\left| \psi_n \left( \operatorname{Avg}_{F_{k_n}} x - \phi(x) 1 \right) \right| = \left\| \operatorname{Avg}_{F_{k_n}} x - \phi(x) 1 \right\|$.
	
	Set
	$$\omega_n = \psi_n \circ \operatorname{Avg}_{F_{k_n}} ,$$
	so $\omega_n (x - \phi(x) 1) = \psi_n \left( \operatorname{Avg}_{F_{k_n}} x - \phi(x) 1 \right)$. Then $(\omega_n)_{n = 1}^\infty$ has a subsequence, call it $(\omega_{n_j})_{j = 1}^\infty$ which converges in the weak*-topology to some $\omega$. This $\omega$ is also a state on $\mathfrak{A}$, and by Lemma \ref{K-B}, we know $\omega$ is $\Theta$-invariant. But $\omega \neq \phi$, since
	\begin{align*}
		\left| \omega(x) - \phi(x) \right|	& = \lim_{j \to \infty} \left| \omega_{n_j}(x) - \phi(x) \right| \\
		& = \lim_{j \to \infty} \left| \omega_{n_j}(x - \phi(x) 1) \right| \\
		& = \lim_{j \to \infty} \left| \psi_{n_j}(\operatorname{Avg}_{F_{k_{n_j}}} x - \phi(x) 1) \right| \\
		& = \lim_{j \to \infty} \left\| \operatorname{Avg}_{F_{k_{n_j}}} x - \phi(x) 1 \right\| \\
		& \geq \epsilon_0 .
	\end{align*}
	This contradicts $(\mathfrak{A}, G, \Theta)$ being uniquely ergodic.
\end{proof}

\begin{Rmk}
	Although \cite[Theorem 5.2]{DuvenhageStroeh} describes conditions under which unique ergodicity of an action of an amenable group on a C*-algebra can be related to the convergence of ergodic averages, that result is not a direct generalization of our Theorem \ref{Unique ergodicity equivalent statements}.
\end{Rmk}

In order to develop the gauge machinery from the previous section in the context of actions of amenable groups, we will need to use slightly different techniques, since we do not have access to the Subadditivity Lemma. The main results of the remainder of this section can be summarized as follows.
\begin{Main results}
	Let $\mathbf{F} = (F_k)_{k = 1}^\infty$ be a right Følner sequence.
	\begin{enumerate}[label=(\alph*)]
		\item Let $(\mathfrak{A}, G, \Theta)$ be a C*-dynamical system, and let $\mathbf{F} = (F_k)_{k = 1}^\infty$ be a right Følner sequence for $G$. Then if $a \in \mathfrak{A}$ is a positive element, then the sequence
		$\left( \left\| \frac{1}{|F_k|} \sum_{g \in F_k} \Theta_g a \right\| \right)_{k = 1}^\infty$ converges to $m \left( a \vert \mathcal{S}^G \right)$.
		\item Let $(\mathfrak{A}, G, \Theta; \iota)$ be a faithful C*-model of $(\mathfrak{M}, \rho, G, \Xi)$. Then the following conditions are related by the implications (i)$\iff$(ii)$\Rightarrow$(iii).
		\begin{enumerate}[label=(\roman*)]
			\item The C*-dynamical system $(\mathfrak{A}, G, \Theta)$ is uniquely ergodic.
			\item The C*-dynamical system $(\mathfrak{A}, G, \Theta)$ is strictly ergodic.
			\item $\Gamma(\iota(a)) = \rho(\iota(a))$ for all positive $a \in \mathfrak{A}$.
		\end{enumerate}
		Further, if $\mathcal{S}^G$ is a simplex, then (iii)$\Rightarrow$(i).
	\end{enumerate}
\end{Main results}

\begin{Thm}\label{Gauge exists for C*-dynamical systems}
	Let $(\mathfrak{A}, G, \Theta)$ be a C*-dynamical system, and let $\mathbf{F} = (F_k)_{k = 1}^\infty$ be a right Følner sequence for $G$. Then if $a \in \mathfrak{A}$ is a positive element, then the sequence
	$\left( \left\| \frac{1}{|F_k|} \sum_{g \in F_k} \Theta_g a \right\| \right)_{k = 1}^\infty$ converges to $m \left( a \vert \mathcal{S}^G \right)$.
\end{Thm}

\begin{proof}	
	For each $k \in \mathbb{N}$, choose a state $\sigma_k$ on $\mathfrak{A}$ such that
	$$\sigma_k \left( \frac{1}{|F_k|} \sum_{g \in F_k} \Theta_g a \right) = \left\| \frac{1}{|F_k|} \sum_{g \in F_k} \Theta_g a \right\| .$$
	Let $\omega_k = \frac{1}{|F_k|} \sum_{g \in F_k} \sigma_k \circ \Theta_g$, so
	\begin{align*}
		\omega_k(x)	& = \frac{1}{|F_k|} \sum_{g \in F_k} \sigma_k (\Theta_g x) \\
		& = \sigma_k \left( \frac{1}{|F_k|} \sum_{g \in F_k} \Theta_g x \right) , \\
		\omega_k(a)	& = \sigma_k \left( \frac{1}{|F_k|} \sum_{g \in F_k} \Theta_g a \right) \\
		& = \left\| \frac{1}{|F_k|} \sum_{g \in F_k} \Theta_g a \right\| .
	\end{align*}
	This means that in order to show that $\left( \left\| \frac{1}{|F_k|} \sum_{g \in G} \Theta_g a \right\| \right)_{k = 1}^\infty$ converges to $m \left( a \vert \mathcal{S}^G \right)$, it suffices to show that $\omega_k(a) \stackrel{k \to \infty}{\to} m \left( a \vert \mathcal{S}^G \right)$. So for the remainder of this proof, we are going to be looking instead at the sequence $(\omega_k)_{k = 1}^\infty$.
	
	Let $k_1 < k_2 < \cdots$ be some sequence such that $\left( \omega_{k_n} \right)_{n = 1}^\infty$ converges in the weak*-topology to some $\omega$. It follows from Lemma \ref{K-B} that $\omega$ is $\Theta$-invariant. To see that \linebreak$(\omega_k(a)_{k = 1}^\infty = \left( \left\| \frac{1}{|F_k|} \sum_{g \in G} \Theta_g \iota(a) \right\| \right)_{k = 1}^\infty$ converges to $m \left( a \vert \mathcal{S}^G \right)$, it will suffice to show that every limit point $\omega$ of $\left( \omega_k : k \in \mathbb{N} \right)$ satisfies 
	$$\omega \in \mathcal{S}_\mathrm{max}^G (a) .$$
	This follows because if there existed a subsequence $k_1 < k_2 < \cdots$ of $(\omega_k)_{k = 1}^\infty$ such that $\omega_{k_n}(a) \stackrel{n \to \infty}{\to} z \neq m \left( a \vert \mathcal{S}^G \right)$, then by compactness, that subsequence $\left( \omega_{k_n} : n \in \mathbb{N} \right)$ would have some subsequence converging to some $\omega'$ for which $\omega'(a) = z \neq m \left( a \vert \mathcal{S}^G \right)$, meaning in particular that $\omega' \not \in \mathcal{S}_\mathrm{max}^G(a)$.
	
	So let $k_1 < k_2 < \cdots$ be some sequence such that $\left( \omega_{k_n} \right)_{n = 1}^\infty$ converges in the weak*-topology to some $\omega$. As has already been remarked, we have that $\omega \in \mathcal{S}^G$, so $\omega(a) \leq m \left( a \vert \mathcal{S}^G \right)$. We prove the opposite inequality. Let $\phi \in \mathcal{S}^G$. Then
	\begin{align*}
		\phi(a)	& = \phi \left( \frac{1}{\left| F_{k_n} \right|} \sum_{g \in F_{k_n} } \Theta_g a \right)	& \left( \textrm{$\phi$ is $\Theta$-invariant} \right) \\
		& \leq \left\| \frac{1}{\left| F_{k_n} \right|} \sum_{g \in F_{k_n}} \Theta_g a \right\| \\
		& = \omega_{k_n}(a) \\
		\Rightarrow \phi(a)	& \leq \lim_{n \to \infty} \omega_{k_n}(a) \\
		& = \omega(a) .
	\end{align*}
	Therefore $\omega(a) \geq \sup_{\psi \in \mathcal{S}^G } \psi(a) = m \left( a \vert \mathcal{S}^G \right)$. This establishes the desired identity.
\end{proof}

\begin{Rmk}
An alternate proof of Theorem \ref{Gauge exists for C*-dynamical systems} using nonstandard analysis is presented in Section \ref{NSA}.
\end{Rmk}

\begin{Cor}\label{Unique ergodicity for C*-dynamical systems in terms of gauge}
Let $(\mathfrak{A}, G, \Theta)$ be a C*-dynamical system, and let $\mathbf{F} = (F_k)_{k = 1}^\infty$ be a right Følner sequence for $G$. Let $\phi \in \mathcal{S}^G$. Then $(\mathfrak{A}, G, \Theta)$ is uniquely ergodic if and only if for every positive element $a \in \mathfrak{A}$, the sequence
$\left( \left\| \frac{1}{|F_k|} \sum_{g \in F_k} \Theta_g a \right\| \right)_{k = 1}^\infty$ converges to $\phi(a)$.
\end{Cor}

\begin{proof}
($\Rightarrow$) Suppose $(\mathfrak{A}, G, \Theta)$ is uniquely ergodic. Then $\phi(a) = m \left( a \vert \mathcal{S}^G \right)$ for all positive $a \in \mathfrak{A}$, so by Theorem \ref{Gauge exists for C*-dynamical systems} $\left\| \frac{1}{|F_k|} \sum_{g \in F_k} \Theta_g a \right\| \stackrel{k \to \infty}{\to} \phi(a)$.

($\Leftarrow$) We'll prove the contrapositive. Suppose $(\mathfrak{A}, G, \Theta)$ is \emph{not} uniquely ergodic. Then there exists an extreme point $\psi$ of $\mathcal{S}^G$ different from $\phi$. By Corollary \ref{Uniquely maximizing states}, there exists $a \in \mathfrak{A}$ self-adjoint such that $\{\phi\} = \mathcal{S}_{\mathrm{max}}^G(a)$. We can assume that $a$ is positive, replacing $a$ by $a + r$ for a sufficiently large positive real number $r > 0$ otherwise. Thus $\lim_{k \to \infty} \left\| \frac{1}{|F_k|} \sum_{g \in F_k} \Theta_g a \right\| = \psi(a) > \phi(a)$.
\end{proof}

\begin{Def}
	Given a C*-dynamical system $(\mathfrak{A}, G, \Theta)$, a positive element $a \in \mathfrak{A}$, and a right Følner sequence $\mathbf{F} = (F_k)_{k = 1}^\infty$ for $G$, we define the \emph{gauge} of $a$ to be the limit
	$$\Gamma(x) : = \lim_{k \to \infty} \left\| \frac{1}{|F_k|} \sum_{g \in F_k} \Theta_g x \right\| .$$
\end{Def}

Theorem \ref{Gauge exists for C*-dynamical systems} shows that the gauge exists, but Theorem \ref{Amenable gamma identity} demonstrates the way that the gauge interacts with a W*-dynamical system and a C*-model. Moreover, the gauge is dependent only on $(\mathfrak{A}, G, \Theta)$, and independent of the right Følner sequence $\mathbf{F} = (F_k)_{k = 1}^\infty$. As such, even though the gauge as we have described it is computed using a right Følner sequence $\mathbf{F} = (F_k)_{k = 1}^\infty$, we do not need to include $\mathbf{F}$ in our notation for $\Gamma$.

\begin{Cor}\label{Ergodic optimization through non-surjective *-homomorphisms}
	Let $\left( \mathfrak{A} , G , \Theta \right) , \left( \tilde{\mathfrak{A}} , G , \tilde{\Theta} \right)$ be two C*-dynamical systems, and let $\pi : \mathfrak{A} \to \tilde{\mathfrak{A}}$ be a *-homomorphism (not necessarily surjective) such that
	\begin{align*}
		\tilde{\Theta}_g \circ \pi	& = \pi \circ \Theta_g	& (\forall g \in G) .
	\end{align*}
	Let $\tilde{\mathcal{S}}^G$ denote the space of $\tilde{\Theta}$-invariant states on $\tilde{\Theta}$. Then $m \left( \pi(a) \vert \tilde{\mathcal{S}}^G \right) = m \left( a \vert \operatorname{Ann}(\ker \pi) \right)$.
\end{Cor}

\begin{proof}
	Let $\mathfrak{B} = \pi(\mathfrak{A})$, and let $H : G \to \operatorname{Aut}(\mathfrak{B})$ be the action $H_g = \tilde{\Theta}_g \vert_{\mathfrak{B}}$. Let $K$ denote the space of all $H$-invariant states on $\mathfrak{B}$. Then
	\begin{align*}
		m \left( \pi(a) \vert \tilde{\mathcal{S}}^G \right)	& = \Gamma_{ \tilde{\mathfrak{A}} } (\pi(a))	& \left( \textrm{Theorem \ref{Gauge exists for C*-dynamical systems}}\right)	\\
		& = \Gamma_{ \mathfrak{B} } (\pi(a)) \\ 
		& = m \left( \pi(a) \vert K \right)	& \left(\textrm{Theorem \ref{Gauge exists for C*-dynamical systems}} \right)\\
		& = m \left( a \vert \operatorname{Ann}(\ker \pi) \right)	& \left(\textrm{Theorem \ref{Ergodic optimization through *-homomorphisms}} \right).
	\end{align*}
	
\end{proof}

\begin{Cor}\label{Amenable gamma identity}
	Let $(\mathfrak{M}, \rho, G, \Xi)$ be a W*-dynamical system, and let $(\mathfrak{A}, G, \Theta; \iota)$ be a C*-model of $(\mathfrak{M}, \rho, \mathbb{Z}, \Xi)$. Then if $a \in \mathfrak{A}$ is a positive element, then
	$$\Gamma(\iota(a)) = m \left( a \vert \operatorname{Ann}(\ker \iota) \right) .$$
\end{Cor}

\begin{proof}
	Write $\tilde{\mathfrak{A}} = \iota(\mathfrak{A}) \subseteq \mathfrak{M}$, and let $\tilde{\Theta} : G \to \operatorname{Aut} \left( \tilde{\mathfrak{A}} \right)$ be the action $\tilde{ \Theta }_g = \Xi_g \vert_{ \tilde{\mathfrak{A}} }$ obtained by restricting $\Xi$ to $\tilde{\mathfrak{A}}$. Write $\tilde{\mathcal{S}}^G$ for the space of $\tilde{\Theta}$-invariant states on $\tilde{\mathfrak{A}}$.
	
	We know $\Gamma_{\mathfrak{M}}(\iota(a)) = \Gamma_{\tilde{\mathfrak{A}}}(\iota(a))$. By Theorem \ref{Gauge exists for C*-dynamical systems}, we know that $\Gamma_{ \tilde{\mathfrak{A}} }(\iota(a)) = m \left( \iota(a) \vert \tilde{\mathcal{S}}^G \right)$, and by Theorem \ref{Ergodic optimization through *-homomorphisms}, we know that
	$$ m \left( \iota(a) \vert \tilde{\mathcal{S}}^G \right) = m \left( a \vert \operatorname{Ann}(\ker \iota) \right) . $$
\end{proof}

This brings us to our characterization of unique ergodicity with respect to the gauge for C*-models.

\begin{Thm}\label{Amenable strict ergodicity and gauge}
	Let $(\mathfrak{M}, \rho, G, \Xi)$ be a W*-dynamical system, and let $(\mathfrak{A}, G, \Theta; \iota)$ be a faithful C*-model of $(\mathfrak{M}, \rho, G, \Xi)$. Then the following conditions are related by the implications (i)$\iff$(ii)$\Rightarrow$(iii).
	\begin{enumerate}[label=(\roman*)]
		\item The C*-dynamical system $(\mathfrak{A}, G, \Theta)$ is uniquely ergodic.
		\item The C*-dynamical system $(\mathfrak{A}, G, \Theta)$ is strictly ergodic.
		\item $\Gamma(a) = \rho(\iota(a))$ for all positive $a \in \mathfrak{A}$.
	\end{enumerate}
	Further, if $\mathcal{S}^G$ is a simplex, then (iii)$\Rightarrow$(i).
\end{Thm}

\begin{proof}
	(i)$\Rightarrow$(ii) Suppose that $(\mathfrak{A}, G, \Theta)$ is uniquely ergodic. Then $\rho \circ \iota$ is an invariant state on $\mathfrak{A}$, so it follows that $\rho \circ \iota$ is \emph{the} unique invariant state. But $\rho \circ \iota$ is also faithful, so it follows that $(\mathfrak{A}, G, \Theta)$ is strictly ergodic.
	
	(ii)$\Rightarrow$(i) Trivial.
	
	(i)$\Rightarrow$(iii) Suppose that $(\mathfrak{A}, G, \Theta)$ is uniquely ergodic, and let $a \in \mathfrak{A}$ be positive. Let $\phi$ be an $\left( a \vert \mathcal{S}^G \right)$-maximizing state on $\mathfrak{A}$. Then $\phi = \rho \circ \iota$, since both are invariant states and $(\mathfrak{A}, G, \Theta)$ is uniquely ergodic. Then $\phi = \rho \circ \iota$, so $\Gamma(a) = \phi(a) = \rho(\iota(a))$.
	
	(iii)$\Rightarrow$(i) Suppose that $\mathcal{S}^G$ is a simplex, but that $(\mathfrak{A}, G, \Theta)$ is \emph{not} uniquely ergodic. Let $\phi \in \mathcal{S}^G$ be an extreme point of $\mathcal{S}^G$ different from $\rho \circ \iota$. Then by Corollary \ref{Uniquely maximizing states}, there exists $a \in \mathfrak{A}$ self-adjoint such that $\{ \phi \} = \mathcal{S}_\mathrm{max}^G(a)$. We can assume that $a$ is positive, since otherwise we could replace $a$ with $a + r$ for a sufficiently large positive real number $r > 0$, and $\mathcal{S}_\mathrm{max}^\mathbb{Z}(a) = \mathcal{S}_\mathrm{max}^\mathbb{Z}(a + r)$. Then $\Gamma(a) = \phi(a)$. But by the assumption that $\phi$ is uniquely $\left( a \vert \mathcal{S}^G \right)$-maximizing, it follows that $\rho(\iota(a)) < \phi(a)$. Therefore $\Gamma(a) \neq \rho(\iota(a))$, meaning that (iii) does not attain. Thus $\neg$(i)$\Rightarrow \neg$(iii).
\end{proof}

\section{A noncommutative Herman ergodic theorem}\label{NC Herman section}

For the duration of this section, we assume that $(\mathfrak{A}, G, \Theta)$ is a C*-dynamical system such that $\mathfrak{A}$ is separable, and that $G$ is amenable.

Let $\mathbf{F} = (F_k)_{k = 1}^\infty$ be a right Følner sequence for $G$. Write $\mathscr{P}^\mathbf{F}(S)$ to denote the set of all limit points of sequences of the form $\left( \phi_k \circ \operatorname{Avg}_{F_k} \right)_{k = 1}^\infty$, where $\phi_k \in S$ for all $k \in \mathbb{N}$. Because $\mathbf{F}$ is right Følner, we know from Lemma \ref{K-B} that if $S$ is nonempty, then $\mathscr{P}^\mathbf{F}(S)$ will be a nonempty compact subset of $\mathcal{S}^G$. In particular, if $S \supseteq \mathcal{S}^G$, then $\mathscr{P}^\mathbf{F}(S) = \mathcal{S}^G$ for any choice of $\mathbf{F}$. Moreover, if $S$ is convex and $\Theta$-invariant, then $\mathscr{P}^\mathbf{F}(S) = \overline{S}$.

\begin{Qn}
Is $\mathscr{P}^\mathbf{F}(S)$ dependent on $\mathbf{F}$ in general?
\end{Qn}

We now define two quantities.

\begin{Not}
Let $\mathbf{F}$ be a right Følner sequence for $G$, and $S$ a nonempty subset of $\mathcal{S}$. Let $x \in \mathfrak{R}$. Define
\begin{align*}
\overline{a}_{\mathbf{F}, S}(x)	& : = \sup \left\{ \psi(x) : \psi \in \mathscr{P}^\mathbf{F}(x) \right\} , \\
\underline{a}_{\mathbf{F}, S}(x)	& : = \inf \left\{ \psi(x) : \psi \in \mathscr{P}^\mathbf{F}(x) \right\} , \\
\overline{d}_{\mathbf{F}, S}(x)	& : = \lim_{k \to \infty} \left( \sup \left\{ \phi \left( \operatorname{Avg}_{F_k} x \right) : \phi \in S \right\} \right) , \\
\underline{d}_{\mathbf{F}, S}(x)	& : = \lim_{k \to \infty} \left( \inf \left\{ \phi \left( \operatorname{Avg}_{F_k} x \right) : \phi \in S \right\} \right) .
\end{align*}
\end{Not}

The values $\overline{a}_{\mathbf{F}, S}, \overline{d}_{\mathbf{F}, S}$ can be compared to the $\alpha$ and $\delta$ quantities presented in Section 2 of \cite{Jenkinson}, respectively. Ergodic optimization is concerned with finding the extrema of sequences of ergodic averages of real-valued functions, but there are several ways we might attempt to formalize what an ``extremum" of a sequence of ergodic averages would be. In \cite{Jenkinson}, O. Jenkinson proposes several different ways we might formalize this notion, then demonstrates that they are equivalent under reasonable conditions \cite[Proposition 2.1]{Jenkinson}. Our Proposition \ref{NC Jenkinson extrema} is an attempt to extend some part of this result to the noncommutative and relative setting.

\begin{Prop}\label{NC Jenkinson extrema}
The quantities $\overline{d}_{\mathbf{F}, S}(x) , \underline{d}_{\mathbf{F}, S}(x)$ are well-defined when $S \subseteq \mathcal{S}^G$ is compact, convex, and $\Theta$-invariant. Moreover, they satisfy
\begin{align*}
\overline{a}_{\mathbf{F}, S}(x)	& = \overline{d}_{\mathbf{F}, S}(x) ,	& \underline{a}_{\mathbf{F}, S}(x)	& = \underline{d}_{\mathbf{F}, S}(x) .
\end{align*}
\end{Prop}

\begin{proof}
We'll prove that $\overline{a}_{\mathbf{F}, S}(x) = \overline{d}_{\mathbf{F}, S}(x)$, as the proof that $\underline{a}_{\mathbf{F}, S}(x) = \underline{d}_{\mathbf{F}, S}(x)$ is very similar. We know a priori that $\mathscr{P}^\mathbf{F}(S) = \overline{S}$.

Let $(\phi_k)_{k = 1}^\infty$ be a sequence in $S$ such that for each $k \in \mathbb{N}$, we have
$$\sup \left\{ \phi \left( \operatorname{Avg}_{F_k} x \right) : \phi \in S \right\} - 1 / k \leq \phi_k (\operatorname{Avg}_{F_k} x) \leq \sup \left\{ \phi \left( \operatorname{Avg}_{F_k} x \right) : \phi \in S \right\} .$$
We know that any limit point of $\left( \phi_k \circ \operatorname{Avg}_{F_k} \right)_{k = 1}^\infty$ is in $\overline{S}$. Let $k_1 < k_2 < \cdots$ be chosen such that $\lim_{\ell \to \infty} \phi_{k_\ell} \left( \operatorname{Avg}_{F_{k_\ell}} x \right) = \limsup_{k \to \infty} \phi_k \left( \operatorname{Avg}_{F_k} x \right)$. We can assume that $\left( \phi_{k_\ell} \circ \operatorname{Avg}_{F_{k_\ell}} \right)_{\ell = 1}^\infty$ is weak*-convergent to a state $\psi \in \overline{S}$, passing to a subsequence if necessary. Then
$$\limsup_{k \to \infty} \phi_k \left( \operatorname{Avg}_{F_k} x \right) = \lim_{\ell \to \infty} \phi_{k_\ell} \left( \operatorname{Avg}_{F_{k_\ell}} x \right) = \psi(x) \leq \overline{a}_{\mathbf{F}, S}(x) .$$

Assume for contradiction that $\liminf_{k \to \infty} \phi_k \left( \operatorname{Avg}_{F_k} x \right) < \overline{a}_{\mathbf{F}, S}(x)$. Let $\psi' \in \mathscr{P}^\mathbf{F}(S)$ be such that $\psi'(x) > \liminf_{k \to \infty} \phi_k \left( \operatorname{Avg}_{F_k} x \right)$. Then
\begin{align*}
\psi'(x)	& = \psi' \left( \operatorname{Avg}_{F_k} x \right)	& \leq \phi_k \left( \operatorname{Avg}_{F_k} x \right) - 1 / k .
\end{align*}
Let $k_1' < k_2' < \cdots$ such that $\left( \phi_{k_\ell'} \left( \operatorname{Avg}_{F_{k_\ell}'} x \right) \right)_{\ell = 1}^\infty$ converges to $\liminf_{k \to \infty} \phi_k \left( \operatorname{Avg}_{F_k} x \right)$. Then
\begin{align*}
\psi'(x)	& \leq \lim_{\ell \to \infty} \phi_{k_\ell'} \left( \operatorname{Avg}_{F_{k_\ell'}} x \right)	& = \liminf_{k \to \infty} \phi_k \left( \operatorname{Avg}_{F_k} x \right)	& < \psi(x) ,
\end{align*}
a contradiction. Therefore we conclude that $\liminf_{k \to \infty} \phi_k \left( \operatorname{Avg}_{F_k} x \right) \geq \overline{a}_{\mathbf{F}, S}(x)$. Thus
$$\overline{a}_{\mathbf{F}, S}(x) \leq \liminf_{k \to \infty} \phi_k \left( \operatorname{Avg}_{F_k} x \right) \leq \limsup_{k \to \infty} \phi_k \left( \operatorname{Avg}_{F_k} x \right) \leq \overline{a}_{\mathbf{F}, S}(x) .$$
Thus we can conclude that $\overline{d}_{\mathbf{F}, S}(x)$ is well-defined and equal to $\overline{d}_{\mathbf{F}, S}(x)$.
\end{proof}

\begin{Rmk}
An alternate proof of Proposition \ref{NC Jenkinson extrema} using nonstandard analysis is presented in Section \ref{NSA}.
\end{Rmk}

To our knowledge, the first result like Theorem \ref{Relative Herman} is \cite[Lemme on pg. 487]{Herman}. Herman's result can be understood as an extension of the classical result that a topological dynamical system is uniquely ergodic if and only if the ergodic averages of all continuous functions converge uniformly to a constant. To our knowledge, the first record of this classical result is \cite[(5.3)]{OxtobyErgodic}. If Oxtoby's result can be understood as relating the uniform convergence properties of ergodic averages of \emph{all} continuous functions to the ergodic optimization of all continuous functions, then Herman's result relates the uniform convergence properties of ergodic averages of a \emph{single} continuous function to its ergodic optimization. Our result extends Herman's in a few directions. First, it extends Herman's result to the setting of actions of amenable groups other than $\mathbb{Z}$. Moreover, it extends the result to C*-dynamical systems. Finally, it allows us to relate convergence in certain seminorms to relative ergodic optimizations.

Let $(\mathfrak{A}, G, \Theta)$ be a C*-dynamical system, where $G$ is an amenable group. Given a nonempty subset $S$ of $\mathcal{S}$, define the seminorm $\| \cdot \|_S$ on $\mathfrak{A}$ by
$$\|x\|_S : = \sup_{\phi \in S} \left| \psi(x) \right| .$$

\begin{Thm}\label{Relative Herman}
Let $\mathbf{F}$ be a right Følner sequence for $G$, and $S \subseteq \mathcal{S}$. Let $x \in \mathfrak{R}$, and $\lambda \in \mathbb{R}$. Then the following are equivalent.
	\begin{enumerate}[label=(\roman*)]
		\item $\left\{ \psi(x) : \psi \in \mathscr{P}^\mathbf{F}(S) \right\} = \{\lambda\}$.
		\item $\lim_{k \to \infty} \left\| \operatorname{Avg}_{F_k} x - \lambda \right\|_S = 0$.
	\end{enumerate}
\end{Thm}

\begin{proof}
	(i)$\Rightarrow$(ii): We prove the contrapositive. Suppose there exists $\epsilon_0 > 0$ and $k_1 < k_2 < \cdots$ such that
	\begin{align*}
		\left\| \operatorname{Avg}_{F_{k_\ell}} x - \lambda \right\|_S	& > \epsilon_0	& (\forall \ell \in \mathbb{N}) .
	\end{align*}
	For each $k \in \mathbb{N}$, choose $\phi_k \in S$ such that $\left| \phi_k \left( \operatorname{Avg}_{F_k} x - \lambda \right) \right| \geq \frac{1}{2} \left\| \operatorname{Avg}_{F_k} x - \lambda \right\|_S$. Then in particular we know that
	\begin{align*}
		\left| \phi_{k_\ell} \left( \operatorname{Avg}_{F_{k_\ell}} x - \lambda \right) \right|	& > \epsilon_0 / 2	& (\forall \ell \in \mathbb{N}) .
	\end{align*}
	By the weak*-compactness of $\mathcal{S}$, there must exist a weak*-convergent subsequence of \linebreak$\left( \phi_{k_\ell} \circ \operatorname{Avg}_{k_\ell} \right)_{\ell = 1}^\infty$. Assume without loss of generality that $\left( \phi_{k_\ell} \circ \operatorname{Avg}_{k_\ell} \right)_{\ell = 1}^\infty$ converges in the weak* topology, and write $\psi = \lim_{\ell \to \infty} \phi_{k_\ell} \circ \operatorname{Avg}_{k_\ell}$. Then
	\begin{align*}
		|\psi(x - \lambda)|	& = \left| \lim_{\ell \to \infty} \phi_{k_\ell} \left( \operatorname{Avg}_{F_{k_\ell}} x - \lambda \right) \right| \\
		& = \lim_{\ell \to \infty} \left| \phi_{k_\ell} \left( \operatorname{Avg}_{F_{k_\ell}} x - \lambda \right) \right| \\
		& \geq \epsilon_0 / 2 .
	\end{align*}
	Therefore $\psi(x) \neq \lambda$, meaning that $\left\{ \psi(x) : \psi \in \mathscr{P}^\mathbf{F}(S) \right\} \neq \{\lambda\}$.
	
	(ii)$\Rightarrow$(i): Suppose that $\lim_{k \to \infty} \left\| \operatorname{Avg}_{F_k} x - \lambda \right\|_S = 0$. Let $(\phi_k)_{k = 1}^\infty$ be a sequence in $S$, and let $\left( \phi_{k_\ell} \circ \operatorname{Avg}_{F_{k_\ell}} \right)_{\ell = 1}^\infty$ be a weak*-convergent subsequence of $\left( \phi_k \circ \operatorname{Avg}_{F_k} \right)_{k = 1}^\infty$ with limit $\psi$. Then
	\begin{align*}
		\left| \psi(x - \lambda) \right|	& = \left| \lim_{\ell \to \infty} \phi_{k_\ell} \left( \operatorname{Avg}_{F_{k_\ell}} x - \lambda \right) \right| \\
		& = \lim_{\ell \to \infty} \left| \phi_{k_\ell} \left( \operatorname{Avg}_{F_{k_\ell}} x - \lambda \right) \right| \\
			& \leq \limsup_{\ell \to \infty} \left\| \operatorname{Avg}_{F_k} x - \lambda \right\|_S \\
		& = 0 .
	\end{align*}
	Therefore $\left\{ \psi(x) : \psi \in \mathscr{P}^\mathbf{F}(S) \right\} = \{\lambda\}$.
\end{proof}

\begin{Rmk}
An alternate proof of Theorem \ref{Relative Herman} using nonstandard analysis is presented in Section \ref{NSA}.
\end{Rmk}

\begin{Cor}\label{NC Global Herman}
Let $\mathbf{F}$ be a right Følner sequence for $G$. Let $x \in \mathfrak{R}$, and $\lambda \in \mathbb{R}$. Then the following are equivalent.
\begin{enumerate}[label=(\roman*)]
	\item $\left\{ \psi(x) : \psi \in \mathcal{S}^G \right\} = \{\lambda\}$.
	\item $\lim_{k \to \infty} \left\| \operatorname{Avg}_{F_k} x - \lambda \right\| = 0$.
\end{enumerate}
\end{Cor}

\begin{proof}
	Apply Theorem \ref{Relative Herman} in the case where $S = \mathcal{S}$, implying that $\|\cdot\|_S = \|\cdot\|$ and $\mathscr{P}^\mathbf{F}(S) = \mathcal{S}^G$.
\end{proof}

Corollary \ref{NC Global Herman} strengthens the noncommutative analogue of Oxtoby's characterization of unique ergodicity, as we see below.

\begin{Cor}[A noncommutative extension of Oxtoby's characterization of unique ergodicity]
Let $\mathbf{F} = (F_k)_{k = 1}^\infty$ be a right Følner sequence for $G$. A C*-dynamical system $(\mathfrak{A}, G, \Theta)$ is uniquely ergodic if and only if $\left(\operatorname{Avg}_{F_k} x \right)_{k = 1}^\infty$ converges in norm to an element of $\mathbb{C}1 \subseteq \mathfrak{A}$ for all $x \in \mathfrak{A}$.
\end{Cor}

\begin{proof}
$(\Rightarrow)$: By taking real and imaginary parts, we can reduce to the case where $x$ is self-adjoint. If $(\mathfrak{A}, G, \Theta)$ is uniquely ergodic, then $\left\{ \psi(x) : \psi \in \mathcal{S}^G \right\}$ is singleton, so by Corollary \ref{NC Global Herman} the averages will converge to a scalar.

$(\Leftarrow)$: Conversely, if $(\mathfrak{A}, G, \Theta)$ is \emph{not} uniquely ergodic, then there exist two states $\psi_1, \psi_2 \in \mathcal{S}^G$ for which there exists $y \in \mathfrak{R}$ such that $\psi_1(y) \neq \psi_2(y)$, implying that $\left\{ \psi(y) : \psi \in \mathcal{S}^G \right\}$ is not singleton. Corollary \ref{NC Global Herman} then tells us that $\left(\operatorname{Avg}_{F_k} x \right)_{k = 1}^\infty$ doesn't converge in norm.
\end{proof}

\section{Applications of nonstandard analysis to noncommutative ergodic optimization}\label{NSA}

The tools of nonstandard analysis can be used to provide alternate proofs of some results in this article. In this section, we assume that the reader is familiar with the basic tools and vocabulary of nonstandard analysis. See \cite{Goldblatt} for references. Since some of the terminology of the field is not entirely universal, we define some of the less universal terms here.

We will assume throughout this section that $(\mathfrak{A}, G, \Theta)$ is a C*-dynamical system, and that $\mathfrak{U}$ is a universe that contains $\mathfrak{A}, G, \mathbb{C}$. Assume that $* : \mathfrak{U} \mapsto \mathfrak{U}'$ is a countably saturated universe embedding. We say that $x \in \prescript{*}{}{\mathbb{C}}$ is \emph{unlimited} if $|x| > n$ for all $n \in \mathbb{N}$, and \emph{limited} otherwise. Let $\mathbb{L} = \bigcup_{n \in \mathbb{N}} \left\{ z \in \prescript{*}{}{\mathbb{C}} : \|z\| \leq n \right\}$ denote the external ring of limited elements of $\prescript{*}{}{\mathbb{C}}$. For $z, w \in \prescript{*}{}{\mathbb{C}}$, we write $z \simeq w$ if $|z - w| < 1 / n$ for all $n \in \mathbb{N}$. This $\simeq$ is an equivalence relation on $\prescript{*}{}{\mathbb{C}}$. We define the \emph{shadow} $\operatorname{sh} : \mathbb{L} \twoheadrightarrow \mathbb{C}$ to be the $\mathbb{C}$-linear functional mapping $z \in \mathbb{L}$ to the unique (standard) complex number $w \in \mathbb{C}$ for which $z \simeq w$. The shadow is also order-preserving on $\mathbb{L} \cap \prescript{*}{}{\mathbb{R}}$. Let $\prescript{*}{}{\mathbb{N}}_\infty : = \left\{ K \in \prescript{*}{}{\mathbb{N}} : \forall k \in \mathbb{N} \; (K \geq k) \right\} = \prescript{*}{}{\mathbb{N}} \setminus \mathbb{N}$ denote the unlimited hypernaturals.

We have the following nonstandard analogue of Lemma \ref{K-B}.

\begin{Lem}\label{Nonstandard K-B}
Let $(\mathfrak{A}, G, \Theta)$ be a C*-dynamical system, and let $G$ be an amenable group. Consider a sequence in $(\phi_k)_{k = 1}^\infty$ in $\mathcal{S}$, and a right Følner sequence $\mathbf{F} = (F_k)_{k = 1}^\infty$ for $G$. Let $K \in \prescript{*}{}{\mathbb{N}}_\infty$ be an unlimited hypernatural, and define a state $\omega : \mathfrak{A} \to \mathbb{C}$ by
$$\omega(x) = \operatorname{sh} \left( \prescript{*}{}{\phi_K \left( \operatorname{Avg}_{F_K} x \right)} \right) .$$
Then $\omega$ is a well-defined $\Theta$-invariant state, and is a limit point of the sequence $\left( \phi_k \circ \operatorname{Avg}_{F_k} \right)_{k = 1}^\infty$.
\end{Lem}

\begin{proof}
First, we take up the well-definedness of $\omega$. If $x \in \mathfrak{A}$, then
$$\forall k \in \mathbb{N} \; \left( \left| \phi_k \left( \operatorname{Avg}_{F_k} x \right) \right| \leq \|x\| \right) ,$$
and so by the Transfer Principle
$$\forall k \in \prescript{*}{}{\mathbb{N}} \; \left( \left| \prescript{*}{}{\phi}_k \left( \operatorname{Avg}_{F_k} x \right) \right| \leq \|x\| \right) .$$
In particular, it follows that $\left| \prescript{*}{}{\phi_K \left( \operatorname{Avg}_{F_K} x \right)} \right| \leq \|x\|$, meaning that $\prescript{*}{}{\phi_K \left( \operatorname{Avg}_{F_K} x \right)} \in \mathbb{L}$. Thus $\omega(x)$ is well-defined. We can similarly prove that $\omega$ is positive and unital by applying the Transfer Principle to the sentences
\begin{align*}
\forall k \in \mathbb{N} \; \forall x \in \mathfrak{A} \;	& \left( \phi_k \left( \operatorname{Avg}_{F_k} \left( x^* x \right) \right) \geq 0 \right) , \\
\forall k \in \mathbb{N} \;	& \left( \phi_k \left( \operatorname{Avg}_{F_k} 1 \right) = 1 \right) .
\end{align*}
To prove the $\Theta$-invariance of $\omega$, we recall from a familiar argument (see proof of Lemma \ref{K-B}) that if $g_0 \in G, x \in \mathfrak{A}$, then
$$\left| \phi_k \left( \operatorname{Avg}_{F_k} \Theta_{g_0} x \right) - \phi_k \left( \operatorname{Avg}_{F_k} x \right) \right| \leq \frac{|F_k g_0 \Delta F_k|}{|F_k|} \left\| x \right\| \stackrel{k \to \infty}{\to} 0 .$$
It follows from a classical result of nonstandard analysis \cite[Theorem 6.1.1]{Goldblatt} that \linebreak$\left| \prescript{*}{}{\phi}_K \left( \operatorname{Avg}_{F_K} \Theta_{g_0} x \right) - \prescript{*}{}{\phi}_K \left( \operatorname{Avg}_{F_K} x \right) \right| \leq \frac{\left| F_K g_0 \Delta F_K \right|}{ \left| F_K \right| } \|x\| \simeq 0$, meaning that $\omega(x) = \omega \left(\Theta_{g_0} x\right)$.

Finally, we argue that $\omega$ is a limit point of $\left( \phi_k \circ \operatorname{Avg}_{F_k} \right)_{k = 1}^\infty$. For $n, \ell, k_0 \in \mathbb{N}; x_1, \ldots, x_\ell \in \mathfrak{A}$, consider the sentence $\sigma_{x_1, \ldots, x_\ell; n, k_0}$ given by
\begin{align*}
\exists k \in \mathbb{N} \;	& \left[ (k \geq k_0) \land \left( \min_{1 \leq j \leq \ell} \left| \omega(x_j) - \phi_k \left( \operatorname{Avg}_{F_k} x_j \right) \right| < 1 / n \right) \right] .
\end{align*}
Then $\prescript{*}{}{\sigma}_{x_1, \ldots, x_\ell; n, k_0}$ is true for all $n, \ell, k_0 \in \mathbb{N}; x_1, \ldots, x_\ell \in \mathfrak{A}$, witnessed by $K$. Therefore, it follows from the Transfer Principle that $\sigma_{x_1, \ldots, x_\ell; n, k_0}$ is true for all $n, \ell, k_0 \in \mathbb{N}; x_1, \ldots, x_\ell \in \mathfrak{A}$. We know that
$$\left\{ \left\{ \psi \in \mathcal{S} : \min_{1 \leq j \leq \ell} |\omega(x_j) - \psi(x_j)| < 1 / n \right\} : n, \ell \in \mathbb{N}; x_1, \ldots, x_\ell \in \mathfrak{A} \right\}$$
is a neighborhood basis for $\omega$ in the weak* topology. Thus we have shown that $\omega$ is a limit point of the sequence $\left( \phi_k \circ \operatorname{Avg}_{F_k} \right)_{k = 1}^\infty$.
\end{proof}

We might ask whether Lemma \ref{Nonstandard K-B} is strictly weaker than Lemma \ref{K-B}, since Lemma \ref{Nonstandard K-B} also asserts that the state it describes is a limit point of the sequence that generates it. In fact, the two lemmas are equivalent in the sense that for a sequence $(\phi_k)_{k = 1}^\infty$ in $\mathcal{S}$, every limit point of the sequence $\left( \phi_k \circ \operatorname{Avg}_{F_k} \right)_{k = 1}^\infty$ can be written as $\operatorname{sh} \left( \prescript{*}{}{\phi_K \left( \operatorname{Avg}_{F_K} x \right)} \right)$ for some $K \in \prescript{*}{}{\mathbb{N}}_\infty$. To see this, choose $k_1 < k_2 < \cdots$ such that $\psi = \lim_{\ell \to \infty} \phi_{k_\ell} \circ \operatorname{Avg}_{F_{k_\ell}}$ exists. Let $\mathcal{N}$ be a countable neighborhood basis for $\psi$ in the weak*-topology, and for each $U \in \mathcal{N}, k \in \mathbb{N}$, let $S_{U, k}$ be the set
$$
S_{U, k} = \left\{ k' \in \mathbb{N} : \left( k' \geq k \right) \land \left( \phi_{k'} \circ \operatorname{Avg}_{F_{k'}} \in U \right) \right\} .
$$
Then $\left\{ S_{U, k} \right\}_{U \in \mathcal{N}, k \in \mathbb{N}}$ has the finite intersection property, and so by the countable saturation of our universe embedding, it follows that there exists $K \in \prescript{*}{}{\mathbb{N}}$ such that
$$K \in \bigcap_{U \in \mathcal{N}, k \in \mathbb{N}} \prescript{*}{}{S}_{U, k} ,$$
which is necessarily unlimited. Then for any $x \in \mathfrak{A}$, we have that $\left| \prescript{*}{}{\phi}_K \left( \operatorname{Avg}_{F_K} x \right) - \psi(x) \right| < 1 / n$ for all $n \in \mathbb{N}$, so
$$\operatorname{sh} \left( \prescript{*}{}{\phi}_K \left( \operatorname{Avg}_{F_K} x \right) \right) = \psi(x) .$$

This correspondence can be generalized in the following result.

\begin{Prop}\label{Unlimited extensions and compactness arguments}
Let $\Omega = (\Omega, \tau)$ be a compact Hausdorff topological space, and let $* : \mathfrak{U} \to \mathfrak{U}'$ be a countably saturated extension of a universe $\mathfrak{U}$ containing $\Omega$ and $\mathbb{N}$. Let $\simeq$ be the equivalence relation on $\prescript{*}{}{\Omega}$ defined by
\begin{align*}
	x \simeq y	& &	&\iff	&	& \forall U \in \tau \; \left( x \in \prescript{*}{}{U} \leftrightarrow y \in \prescript{*}{}{U} \right) .
\end{align*}
Define a map $\operatorname{sh} : \prescript{*}{}{\Omega} \to \Omega$ that sends $x \in \prescript{*}{}{\Omega}$ to the unique $y \in \Omega$ such that $x \simeq y$, and let $(x_k)_{k = 1}^\infty$ be a sequence in $\Omega$. Then the map $\operatorname{sh}$ is well-defined.

Further, set
\begin{align*}
	& \operatorname{LS} \left( (x_k)_{k = 1}^\infty \right) \\
	=	& \left\{ \omega \in \Omega : \forall U \in \tau \; \forall k \in \mathbb{N} \; \left[ \left( \omega \in U \right) \rightarrow \left( \exists k' \in \mathbb{N} \; \left( \left( k' \geq K \right) \land \left( x_{k'} \in U \right) \right) \right) \right] \right\} .
\end{align*}
Then
$$\left\{ \operatorname{sh} \left( \prescript{*}{}{x}_K \right) : K \in \prescript{*}{}{\mathbb{N}}_\infty \right\} \subseteq \operatorname{LS} \left((x_k)_{k = 1}^\infty \right) .$$
In addition, if $*$ is $\kappa$-saturated for some uncountable cardinal $\kappa > |\mathcal{B}|$, where $\mathcal{B}$ is some topological basis $\mathcal{B}$ of $\tau$, then $\left\{ \operatorname{sh} \left( \prescript{*}{}{x}_K \right) : K \in \prescript{*}{}{\mathbb{N}}_\infty \right\} = \operatorname{LS} \left((x_k)_{k = 1}^\infty \right) .$
\end{Prop}

\begin{proof}
The fact that in a compact topological space, for every $x \in \prescript{*}{}{\Omega}$ exists some $y \in \Omega$ such that $x \simeq y$ can be found in \cite[Theorem 1.6 of Chapter 3]{AppliedNSA}. As for the uniqueness, assume for contradiction that there existed $y, z \in \Omega$ such that $x \simeq y \simeq z, y \neq z$. Using the Hausdorff property, choose open neighborhoods $U, V$ such that $y \in U, z \in V, U \cap V = \emptyset$. Then $y \in U$, so by the Transfer Principle $y \in \prescript{*}{}{U}$. Thus $z \in \prescript{*}{}{U}$, because $y \simeq z$. Therefore $z \in U$ by the Transfer Principle, a contradiction. Thus $\operatorname{sh} : \prescript{*}{}{\Omega} \to \Omega$ is well-defined.

Let $K \in \prescript{*}{}{\mathbb{N}}_\infty$, and consider $y = \operatorname{sh} \left( \prescript{*}{}{x}_K \right)$. Let $\mathcal{N}_y = \left\{ U \in \mathcal{B} : y \in U \right\}$, where $\mathcal{B}$ is a topological basis for $\tau$, and consider for $k \in \mathbb{N}, U \in \mathcal{N}_y$ the sentence $\sigma_{U, k}$ defined by
$$\exists k' \in \mathbb{N} \; \left[ \left( k' \geq k \right) \land \left( x_k \in U \right) \right] .$$
Then $\prescript{*}{}{\sigma}_{k, U}$ is true for all $k \in \mathbb{N}, U \in \mathcal{N}_y$, since $\prescript{*}{}{x}_K \in \prescript{*}{}{U}$ and $K \geq k$ for all $k \in \mathbb{N}$, so it follows that $\sigma_{k, U}$ is true for all $k \in \mathbb{N}, U \in \mathcal{N}_y$. Since $\mathcal{N}_y$ forms a neighborhood basis for $y$, it follows that $y \in \operatorname{LS} \left( (x_k)_{k = 1}^\infty \right)$.

Now suppose that $*$ is $\kappa$-saturated for some uncountable cardinal $\kappa > |\mathcal{B}|$, and let $\omega \in \operatorname{LS} \left( (x_k)_{k = 1}^\infty \right)$. Let $\mathcal{N}_\omega = \left\{ U \in \mathcal{B} : \omega \in U \right\}$. For $k \in \mathbb{N}, U \in \mathcal{N}_\omega$, consider the set
$$S_{k, U} = \left\{ k' \in \mathbb{N} : \left( k' \geq k \right) \land \left( x_{k'} \in U \right) \right\} .$$
Then $\left\{ S_{k, U} : k \in \mathbb{N}, U \in \mathcal{N}_\omega \right\}$ has the finite intersection property, and thus there exists $K \in \bigcap_{k \in \mathbb{N}, U \in \mathcal{N}_\omega} \prescript{*}{}{S}_{k, U}$. Thus $\prescript{*}{}{x}_K \in \prescript{*}{}{U}$ for all $U \in \mathcal{N}_\omega$, and $K \in \prescript{*}{}{\mathbb{N}}_\infty$. Thus $\omega = \operatorname{sh} \left( \prescript{*}{}{x}_K \right)$.
\end{proof}

\begin{Rmk}
Our definitions of $\simeq$ and $\operatorname{sh}$ in the statement of Proposition \ref{Unlimited extensions and compactness arguments} is consistent with our definition of $\simeq$ on $\mathbb{L}$ in the following sense. We can write $\mathbb{L} = \bigcup_{n \in \mathbb{N}} \prescript{*}{}{\left\{ z \in \mathbb{C} : |z| \leq n \right\}}$. If $x, y \in \mathbb{L}$, then there exists $n \in \mathbb{N}$ such that $\max \{|x|, |y|\} \leq n$. Then $x, y \in \prescript{*}{}{\left\{ z \in \mathbb{C} : |z| \leq n \right\} }$. The set $\left\{ z \in \mathbb{C} : |z| \leq n \right\}$ is compact, and the definition of $\simeq$ on that compact space in the sense of Proposition \ref{Unlimited extensions and compactness arguments} will agree with our definition of $\simeq$ on $\mathbb{L}$ from the start of this section.
\end{Rmk}

In light of Theorem \ref{Unlimited extensions and compactness arguments}, several compactness arguments in this article can be proven alternatively in the language of nonstandard analysis. Here we provide a few examples.

\begin{proof}[Proof of Theorem \ref{Gauge exists for C*-dynamical systems} using nonstandard analysis]
For each $k \in \mathbb{N}$, choose a state $\phi_k$ on $\mathfrak{A}$ such that
$$\phi_k \left( \operatorname{Avg}_{F_k} a \right) = \left\| \operatorname{Avg}_{F_k} a \right\| .$$
Fix $K \in \prescript{*}{}{\mathbb{N}}_\infty$, and let $\omega : \mathfrak{A} \to \mathbb{C}$ be the state
$$\omega(x) = \operatorname{sh} \left( \prescript{*}{}{\phi}_K \left( \operatorname{Avg}_{F_K} x \right) \right) .$$
Lemma \ref{Nonstandard K-B} tells us that $\omega$ is $\Theta$-invariant. We argue now that $\omega(a) = m \left( a \vert \mathcal{S}^G \right)$. This follows because if $\psi \in \mathcal{S}^G$, then
$$
\psi(a) = \psi \left( \operatorname{Avg}_{F_k} a \right) \leq \left\| \operatorname{Avg}_{F_k} a \right\| = \phi_k \left( \operatorname{Avg}_{F_k} a \right)
$$
for all $k \in \mathbb{N}$, and thus we can apply the Transfer Principle to the sentence \linebreak$\forall k \in \mathbb{N} \; \left( \psi(a) \leq \phi_k \left( \operatorname{Avg}_{F_k} a \right) \right)$ to infer
$$\psi(a) \leq \prescript{*}{}{\phi}_K \left( \operatorname{Avg}_{F_K} a \right) \Rightarrow \psi(a) \leq \omega(a) .$$

Therefore, we've proven that $\prescript{*}{}{\left\| \operatorname{Avg}_{F_K} a \right\|} \simeq m \left( a \vert \mathcal{S}^G \right)$ for all $K \in \prescript{*}{}{\mathbb{N}}_\infty$. Therefore by a classical result of nonstandard analysis \cite[Theorem 6.1.1]{Goldblatt}, it follows that $\lim_{k \to \infty} \left\| \operatorname{Avg}_{F_k} a \right\| = m \left( a \vert \mathcal{S}^G \right)$.
\end{proof}

\begin{proof}[Proof of Proposition \ref{NC Jenkinson extrema} using nonstandard analysis]
We'll prove that $\overline{a}_{\mathbf{F}, S}(x) = \overline{d}_{\mathbf{F}, S}(x)$, as the proof that $\underline{a}_{\mathbf{F}, S}(x) = \underline{d}_{\mathbf{F}, S}(x)$ is very similar. We know a priori that $\mathscr{P}^\mathbf{F}(S) = \overline{S}$.

Let $(\phi_k)_{k = 1}^\infty$ be a sequence in $S$ such that for each $k \in \mathbb{N}$, we have
$$\sup \left\{ \phi \left( \operatorname{Avg}_{F_k} x \right) : \phi \in S \right\} - 1 / k \leq \phi_k (\operatorname{Avg}_{F_k} x) \leq \sup \left\{ \phi \left( \operatorname{Avg}_{F_k} x \right) : \phi \in S \right\} .$$
Let $K \in \prescript{*}{}{\mathbb{N}}_\infty$, and let $\omega : \mathfrak{A} \to \mathbb{C}$ be the state $\omega (y) = \operatorname{sh} \left( \prescript{*}{}{\phi}_K \left( \operatorname{Avg}_{F_K} y \right) \right)$. Then $\omega \in \mathscr{P}^\mathbf{F}(S)$, so $\omega(x) \leq \overline{a}_{\mathbf{F}, S}(x)$.

To prove the opposite inequality, let $\psi \in \mathscr{P}^\mathbf{F}(S) = \overline{S}$. Then
\begin{align*}
\psi(x)	& = \psi \left( \operatorname{Avg}_{F_k} x \right) \\
	& \leq \sup \left\{ \phi \left( \operatorname{Avg}_{F_k} x \right) : \phi \in S \right\} \\
	& \leq \phi_k \left( \operatorname{Avg}_{F_k} x \right) + 1/k	& (\forall k \in \mathbb{N}) .
\end{align*}
Thus the sentence
$$\forall k \in \mathbb{N} \; \left( \psi(x) \leq \phi_k\left( \operatorname{Avg}_{F_k} x \right) + 1/k \right)$$
is true. Applying the Transfer Principle then tells us that $\psi(x) \leq \prescript{*}{}{\phi}_K \left( \operatorname{Avg}_{F_K} x \right) + 1/K$, implying that $\psi(x) \leq \omega(x)$. Taking a supremum over $\psi \in \overline{S} = \mathscr{P}^\mathbf{F}(S)$ tells us that
$$\overline{a}_{\mathbf{F}, S}(x) \leq \omega(x) .$$

Therefore $\prescript{*}{}{\phi}_K(x) \simeq \overline{a}_{\mathbf{F}, S}(x)$ for all $K \in \prescript{*}{}{\mathbb{N}}$. Thus, by a classical result of nonstandard analysis \cite[Theorem 6.1.1]{Goldblatt}, it follows that $\lim_{k \to \infty} \left\| \operatorname{Avg}_{F_k} a \right\| = \overline{a}_{\mathbf{F}, S}(x)$.
\end{proof}

\begin{proof}[Proof of Theorem \ref{Relative Herman} using nonstandard analysis]
(i)$\Rightarrow$(ii): Suppose $\left\{ \psi(x) : \psi \in \mathscr{P}^\mathbf{F}(S) \right\} = \{\lambda\}$. For each $k \in \mathbb{N}$, choose $\phi_k \in S$ such that $\left| \phi_k \left( \operatorname{Avg}_{F_k} x - \lambda \right) \right| \geq \frac{1}{2} \left\| \operatorname{Avg}_{F_k} x - \lambda \right\|_S$. Fix $K \in \prescript{*}{}{\mathbb{N}}$, and let $\omega : \mathfrak{A} \to \mathbb{C}$ be the state
$$\omega(y) = \operatorname{sh} \left( \prescript{*}{}{\phi}_K \left( \operatorname{Avg}_{F_K} y \right) \right) .$$
Lemma \ref{Nonstandard K-B} tells us that $\omega \in \mathscr{P}^\mathbf{F}(S)$. Thus $\omega(x) = \lambda$. Therefore $\left| \prescript{*}{}{\phi}_K \left( \operatorname{Avg}_{F_K} x \right) - \lambda \right| \simeq 0$ for all $K \in \prescript{*}{}{\mathbb{N}}_\infty$, meaning a classical result of nonstandard analysis \cite[Theorem 6.1.1]{Goldblatt} tells us that $\lim_{k \to \infty} \left| \phi_k \left( \operatorname{Avg}_{F_k} x \right) - \lambda \right| = 0$. But because $\left\| \operatorname{Avg}_{F_k} x - \lambda \right\|_S \leq 2 \left| \phi_k \left( \operatorname{Avg}_{F_k} x \right) - \lambda \right|$ for all $k \in \mathbb{N}$, we can conclude that $\lim_{k \to \infty} \left\| \operatorname{Avg}_{F_k} x - \lambda \right\|_S = 0$.

(ii)$\Rightarrow$(i): Suppose that $\lim_{k \to \infty} \left\| \operatorname{Avg}_{F_k} x - \lambda \right\|_S = 0$.  Let $(\phi_k)_{k = 1}^\infty$ be a sequence in $S$, and let $\omega : \mathfrak{A} \to \mathbb{C}$ be the state
$$\omega(y) = \operatorname{sh} \left( \prescript{*}{}{\phi}_K \left( \operatorname{Avg}_{F_K} y \right) \right) .$$
Then
$$\left| \omega(x - \lambda) \right| \simeq \left| \prescript{*}{}{\phi}_K \left( \operatorname{Avg}_{F_K} x - \lambda \right) \right| \leq \left\| \prescript{*}{}{\operatorname{Avg}}_{F_K} x - \lambda \right\|_S \simeq 0 .$$
Therefore $\omega(x) = \lambda$. We can then take a supremum to get
$$\sup_{ \left(\phi_k\right)_{k = 1}^\infty \in S^\mathbb{N}, K \in \prescript{*}{}{\mathbb{N}}_\infty } \left| \operatorname{sh} \left( \prescript{*}{}{\phi}_K \left( \operatorname{Avg}_{F_K} x \right) \right) - \lambda \right| = 0 .$$
But in light of Proposition \ref{Unlimited extensions and compactness arguments}, we know that
$$\mathscr{P}^\mathbf{F}(S) = \left\{ y \mapsto \operatorname{sh} \left( \prescript{*}{}{\phi}_K \left( \operatorname{Avg}_{F_K} y \right) \right) : (\phi_k)_{k = 1}^\infty \in S^\mathbb{N}, K \in \prescript{*}{}{\mathbb{N}}_\infty \right\} ,$$
so this shows that $\psi(x) = \lambda$ for all $\psi \in \mathscr{P}^\mathbf{F}(S)$.
\end{proof}

\section*{Acknowledgments}

This paper is written as part of the author's graduate studies. He is grateful to his beneficent advisor, professor Idris Assani, for no shortage of helpful guidance.

An earlier version of this paper referred to ``tempero-spatial differentiations." Professor Mark Williams pointed out that the more correct portmanteau would be ``temporo-spatial." We thank Professor Williams for this observation.

\bibliography{Bibliography}
\end{document}